\documentclass[a4paper,reqno,11pt]{amsart}
\usepackage[utf8]{inputenc}
\usepackage{hyperref,xy,doi}
\usepackage{amsmath}
\usepackage{amsthm}
\usepackage{amssymb}
\usepackage{mathrsfs}
\usepackage{graphicx}
\usepackage{mathtools}
\usepackage{enumerate}
\usepackage{multicol}
\usepackage{multirow}
\usepackage{xcolor}
\usepackage[parfill]{parskip}
\usepackage{stix}
\usepackage{csquotes}
\usepackage{soul}
\newcommand{\bitrace}{\text{bitrace}}
\usepackage{setspace}
\numberwithin{equation}{subsection}

\xyoption{all}
\usepackage{ytableau,youngtab}

\usepackage{tikz}
\usepackage{tikz-cd}
\usetikzlibrary{positioning, arrows, patterns,shapes,backgrounds, decorations.pathreplacing}

\hypersetup{
	colorlinks   = true, 
	urlcolor     = blue, 
	linkcolor    = purple, 
	citecolor   = blue 
}
\usepackage[margin= 1.1 in]{geometry}

\newtheorem{theorem}{Theorem}[section]
\newtheorem{corollary}[theorem]{Corollary}
\newtheorem{lemma}[theorem]{Lemma}
\newtheorem{proposition}[theorem]{Proposition}
\theoremstyle{definition}

\theoremstyle{remark}

\newtheorem{example}[theorem]{Example}

\DeclareMathOperator{\Res}{Res}
\DeclareMathOperator{\Ind}{Ind}
\DeclareMathOperator{\triv}{triv}

\DeclareMathOperator{\End}{End}

\newcommand{\C}{ \mathbb C}

\makeatletter
\@namedef{subjclassname@2020}{%
	\textup{2020} Mathematics Subject Classification}
\makeatother

\title[Semi-simple partition algebras as centralizers  for rook  monoids]{Semi-simple 
partition algebras as centralizers  of 
representations for rook monoids}

\author[Volodymyr Mazorchuk]{Volodymyr Mazorchuk}
\address{Uppsala Universitet, Sweden}
\email{mazor@math.uu.se}

\author[Shraddha Srivastava]{Shraddha Srivastava}
\address{Indian Institute of Technology, Dharwad, 580011}
\email{maths.shraddha@gmail.com}

\date{\today}

\begin{document}

\begin{abstract}
Let $\mathcal{P}_k(\delta)$, where $k$ is a positive integer and $\delta$  
some complex parameter, be the classical partition algebra over 
the complex numbers. In the case when $\delta=n$,
it is well-known that the algebra $\mathcal{P}_k(\delta)$  
is the centralizer of the symmetric group $S_n$ acting on the 
$k$-fold tensor space of the natural representation of $S_n$, 
for $n\geq 2k$. The algebra $\mathcal{P}_k(\delta)$ is semi-simple
for generic values of $\delta$. In this paper, we show that 
semi-simple partition algebras appear as the centralizer algebras 
for certain representations of the rook monoids given by an 
iterative  restriction-induction of the trivial representation. 
Along the way, we also give a decomposition of this 
iterative representation of the rook monoid into various 
tensor spaces and show that the corresponding dimensions are given by 
generalized Bell numbers. 
\end{abstract}

\subjclass[2020]{Primary: 18M05, 05E05; Secondary: 05A18, 20M30}
\keywords{partition algebra;  dual symmetric inverse monoid; rook  monoid;
representation; Bell number; generalized Bell number }
\maketitle

\section{Introduction}

The classical Schur--Weyl duality between the symmetric and the general linear groups,
see \cite{Schur,Weyl},
is a cornerstone of modern representation theory. Over the years, many variations
and generalizations of Schur--Weyl duality have been established, leading to 
emergence of the so-called diagram algebras. One such classical generalization
is the Schur--Weyl duality between the symmetric group and the partition algebra,
a prototypical example of diagram algebras, independently discovered by Jones and Martin,
see \cite{Jones,Martin}.

The partition algebra $\mathcal{P}_k(\delta)$, when defined  over the complex  
numbers,  depends on two parameters, namely, a positive integer  $k$ and 
a complex number $\delta$. The algebra $\mathcal{P}_k(\delta)$ has a
distinguished basis formed by all set-partitions of a set consisting of 
$2k$ elements. These set partitions are usually drawn as partition diagrams
(with $k$ elements in the top row and $k$ elements in the bottom row).
The parameter $\delta$ comes into play in the multiplication rule: when 
multiplying two set partitions, they are concatenated in a certain way 
and the outcome is transformed into a new set partition, up to some
possible ``redundant'' parts. The number of these parts determines the
power $\delta$ that appears as a scalar factor in the product.
Having a basis consisting of set partitions, the dimension of 
$\mathcal{P}_k(\delta)$ is given by the Bell number $B(2k)$, which counts
the number of such set partitions. 

It is well-known, see \cite{Martin96,13}, that
the algebra $\mathcal{P}_k(\delta)$ is semi-simple provided that 
the parameter $\delta$ satisfies
$\delta\not\in\{0,1,2,\ldots, 2k-2\}$. Consequently, 
if $\delta = n\geq 2k$, then the algebra $\mathcal{P}_k(n)$
is isomorphic to the endomorphism algebra of 
$V_n^{\otimes k}$, where $V_n$ is the natural representation 
$\mathbb{C}^n$ of the symmetric group $S_n$. 

This natural representation $V_n$ admits an interesting functorial
interpretation. Consider the category $S_n$-mod of all finite-dimensional
$S_n$-modules. We can view $S_{n-1}$ as a subgroup of $S_n$ in the usual
way. Then we have the usual pair of biadjoint functors, the 
restriction $\Res^{S_n}_{S_{n-1}}$ and the induction 
$\Ind^{S_n}_{S_{n-1}}$, between the categories $S_n$-mod and $S_{n-1}$-mod:
\begin{displaymath}
\xymatrix{
S_n\text{-}\mathrm{mod}\ar@/^3mm/[rrrr]^{\Res^{S_n}_{S_{n-1}}}&&&&
S_{n-1}\text{-}\mathrm{mod}
\ar@/^3mm/[llll]^{\Ind^{S_n}_{S_{n-1}}}
}
\end{displaymath}
In particular, we have the endofunctor
$\mathrm{F}$ of $S_n$-mod given by the composition
$\Ind^{S_n}_{S_{n-1}}\circ \Res^{S_n}_{S_{n-1}}$.
It turns out that the functor $\mathrm{F}$ is isomorphic to 
the functor $V_n\otimes_{\mathbb{C}}{}_-$ of tensoring with
$V_n$. In particular,
the $k$-th power $\mathrm{F}^{k}$ is isomorphic to
tensoring with $V_n^{\otimes k}$. By evaluating such an isomorphism
at the trivial $S_n$-module $\mathbb{C}_{\triv}$,
tensoring with which is isomorphic to the 
identity endofunctor of $S_n$-mod, we obtain
\begin{displaymath}
V_n^{\otimes k}\cong 
(\Ind^{S_n}_{S_{n-1}}\circ\Res^{S_n}_{S_{n-1}})^k(\mathbb{C}_{\triv})
=\mathrm{F}^k(\mathbb{C}_{\triv}).
\end{displaymath}

Let $R_n$ denote the set consisting of partial permutation matrices of 
size $n\times n$. With respect to matrix multiplication,  $R_n$ 
becomes a monoid, known as the rook monoid (also known as the
symmetric inverse semigroup or monoid), see \cite{Grood,Solomon,GM,Wagner}. 
There is also a dual object
(even in a certain categorical sense, see \cite{KMb}), 
called the dual symmetric inverse monoid
and denoted $I_n^*$,
introduced in \cite{EL93,EL95}. This monoid is a submonoid of 
$\mathcal{P}_n(\delta)$ (regardless of $\delta$)
that consists of all partition diagrams 
in which each part has a non-trivial intersection with both the top 
and the bottom rows of the diagram. Such parts are called propagating,
so $I_n^*$ is exactly the set of all partition diagrams in
which all parts are propagating. 
Both $R_n$ and $I_n^*$ generalize $S_n$ and both are inverse monoids,
which is a natural class of monoids that are very close to groups.
There is a Schur-Weyl duality relating $R_n$ with $I_k^*$,
established in \cite{5}. The monoids $R_n$ and $I_n^*$
are intensively studied, see \cite{FL98,Pa,11,9,MS25,AM,KM09} 
and references therein.
They are a rich source of interesting
examples and phenomena. In the present paper we will show that the 
representation theory of these monoids is also closely connected to the
representation theory of partition algebras.

Just like the case of symmetric groups, the rook monoid $R_{n-1}$
is, naturally, a submonoid of $R_n$. In particular, the monoid algebra
of $R_{n-1}$ is a unital subalgebra of the monoid algebra of $R_n$.
Therefore we again can talk about induction and restriction between
the corresponding module categories. The starting point for the present
paper was the observation that, iterating the composition of 
these restriction and induction and applying it to the trivial 
representation of $R_n$ produces precisely the Bratteli diagram for 
the multiplicity-free  tower of partition algebras, which can be found in \cite{13}. 

Let us describe this in more detail.
The irreducible representations of $S_n$ over the complex numbers are
indexed by partitions of $n$, see \cite{Sa}. For $\lambda\vdash n$, we have the
corresponding Specht module ${S}^\lambda$. The classical
branching rule asserts that, for $\lambda\vdash n$, 
the restriction of ${S}^\lambda$ to $S_{n-1}$
is a multiplicity free direct sum of those ${S}^\mu$,
where $\mu\vdash n-1$, for which
$\mu$ is obtained from $\lambda$ by removing a removable node in the
Young diagram of $\lambda$. 

The irreducible representations of $R_n$ are indexed by all 
partitions of all $i$, where $0\leq i\leq n$.  For $\lambda\vdash i$,
we denote the corresponding simple module by $R_n^\lambda$. 
The trivial representation  of $R_n$ is $R^{\varnothing}_{n}$, it
corresponds to the trivial partition of $i=0$ and the associated 
Young diagram is usually denoted by $\varnothing$. The branching rule
for $R_n$ is very similar to the one for $S_n$. Restricting 
$R_n^\lambda$ to $R_{n-1}$ is again multiplicity free and, for any
$\mu$ obtained from $\lambda$ be removing a removable node, 
the corresponding $R_{n-1}^\mu$ appears as a summand. The only 
difference is that, if $\lambda\vdash i<n$, then 
$R_{n-1}^\lambda$ appears as well.

Now it is easy to write down the Bratelli diagram for 
an iterated restriction-induction for rook monoids applied to the
trivial module, see Figure~\ref{fig1}, where each simple
module is represented by its Young diagram. Comparing with 
the  Bratteli diagram for  the multiplicity-free  tower of 
partition algebras, one naturally arrives to the question 
whether it is indeed the partition algebra that controls
the endomorphisms of the obtained module. The main aim of
the present paper is to give a positive answer to 
this question.

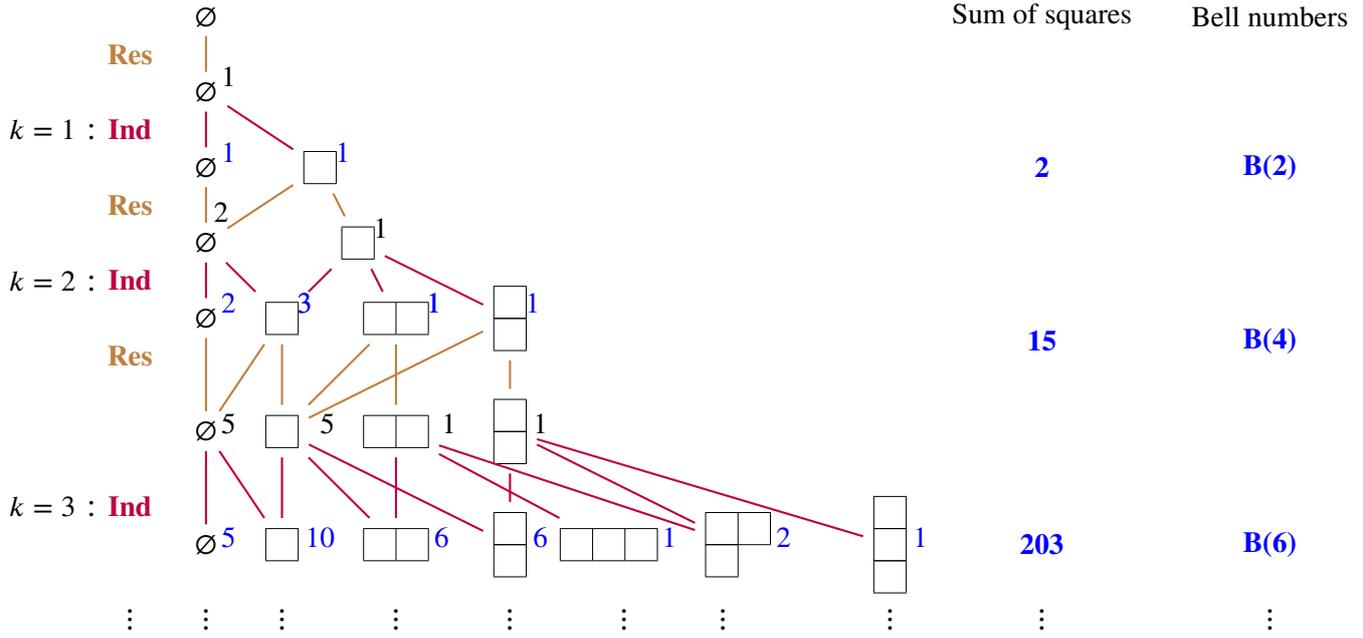
\begin{figure}
\centering
\begin{tikzpicture}
    \node (n1) at (0,0) {$\varnothing$};
 \node (n2) at (0,-1) {$\varnothing$};
    \node at (-1,-0.5) {\textcolor{brown}{\textbf{Res}}};
  
 \node at (0.3,-0.8) {1};
    \node (n3) at (0,-2) {$\varnothing$};
   \node at (0.3,-1.8) {\textcolor{blue}{\textbf{$1$}}};
    \node (n4) at (1.5,-2) {$\yng(1)$};
  \node at (1.8,-1.8) {\textcolor{blue}{\textbf{$1$}}};
   \node at (-1,-1.5) {\textcolor{purple}{\textbf{Ind}}};
    \node at (-2,-1.5) {$k=1:$};
    \node (n5) at (0, -3) {$\varnothing$};
\node at (0.2,-2.6) {2};
    \node (n6) at (2,-3) {$\yng(1)$};
\node at (2.3,-2.8) {1};
    \node at (-1,-2.5) {\textcolor{brown}{\textbf{Res}}};
  \node (n7) at (0,-4) {$\varnothing$};
  \node at (0.3,-3.8) {\textcolor{blue}{\textbf{$2$}}};
    \node (n8) at (1,-4) {$\yng(1)$};
   \node at (1.3,-3.8) {\textcolor{blue}{\textbf{$3$}}};
    \node (n9) at (2.5,-4) {$\yng(2)$};
\node at (3,-3.8) {\textcolor{blue}{\textbf{$1$}}};
\node (n10) at (4,-4) {$\yng(1,1)$};
\node at (4.3,-3.8) {\textcolor{blue}{\textbf{$1$}}};
\node at (-1,-3.5) {\textcolor{purple}{\textbf{Ind}}};
\node at (-2,-3.5) {$k=2:$};
\node (n11) at (0,-5.5) {$\varnothing$};
\node (n12) at (1,-5.5) {$\yng(1)$};
\node (n13) at (2.5,-5.5) {$\yng(2)$};
\node (n14) at (4,-5.5) {$\yng(1,1)$};
\node at (-1,-4.5) {\textcolor{brown}{\textbf{Res}}};
\node at (0.3,-5.4) {$5$};
\node at (1.6,-5.4) {$5$};
\node at (3.2,-5.4) {$1$};
\node at (4.4,-5.4) {$1$};

 \node (n15) at (0,-7) {$\varnothing$};
  \node (n16) at (1,-7) {$\yng(1)$};
   \node at (1.3,-3.8) {\textcolor{blue}{\textbf{$3$}}};
    \node (n17) at (2.5,-7) {$\yng(2)$};
\node at (3,-3.8) {\textcolor{blue}{\textbf{$1$}}};
\node (n18) at (4,-7) {$\yng(1,1)$};
\node (n19) at (5.3,-7) {$\yng(3)$};
\node (n20) at (7,-7) {$\yng(2,1)$};
 \node (n21) at (9,-7) {$\yng(1,1,1)$};

 \node at (0.3,-6.9) {\textcolor{blue}{\textbf{$5$}}};
\node at (1.5,-6.9) {\textcolor{blue}{\textbf{$10$}}};
\node at (3.1,-6.9) {\textcolor{blue}{\textbf{$6$}}};
\node at (4.4,-6.9) {\textcolor{blue}{\textbf{$6$}}};
\node at (6.1,-6.9) {\textcolor{blue}{\textbf{$1$}}};
\node at (7.6,-6.9) {\textcolor{blue}{\textbf{$2$}}};
\node at (9.4,-6.9) {\textcolor{blue}{\textbf{$1$}}};

  \node at (-1,-8) {$\vdots$};
     \node at (0,-8) {$\vdots$};
    \node at (1,-8) {$\vdots$};
\node at (2.5,-8) {$\vdots$};
\node at (4,-8) {$\vdots$};
\node at (5.5,-8) {$\vdots$};
\node at (6.8,-8) {$\vdots$};
\node at (9,-8) {$\vdots$};
\node at (11,-8) {$\vdots$};
\node at (14,-8) {$\vdots$};

\node at (-1,-6.5) {\textcolor{purple}{\textbf{Ind}}};
\node at (-2,-6.5) {$k=3:$};

\draw[brown, thick] (n7)--(n11);
 \draw[brown, thick] (n8)--(n11);
 \draw[brown, thick] (n8)--(n12);
 \draw[brown, thick] (n9)--(n12);
 \draw[brown, thick] (n10)--(n12);
 \draw[brown, thick] (n9)--(n13);
 \draw[brown, thick] (n10)--(n14);

 \draw[purple, thick] (n11)--(n15);
 \draw[purple, thick] (n11)--(n16);
 \draw[purple, thick] (n12)--(n16);
  \draw[purple, thick] (n12)--(n17);
   \draw[purple, thick] (n12)--(n18);
 \draw[purple, thick] (n13)--(n17);
 \draw[purple, thick] (n13)--(n19);
 \draw[purple, thick] (n13)--(n20);
 \draw[purple, thick] (n14)--(n18);
 \draw[purple, thick] (n14)--(n20);
 \draw[purple, thick] (n14)--(n21);

  \draw[brown,thick] (n1) -- (n2);    
  \draw[purple,thick] (n2) -- (n3);
  \draw[purple,thick] (n2) -- (n4);
  \draw[brown,thick] (n3) -- (n5);
  \draw[brown,thick] (n4) -- (n6);
  \draw[brown, thick](n4)--(n5);
  \draw[purple,thick] (n5) -- (n7);
  \draw[purple,thick] (n5) -- (n8);
  \draw[purple,thick] (n6) -- (n8);
   \draw[purple,thick] (n6) -- (n9);
    \draw[purple,thick] (n6) -- (n10);
 \node at (11,0) {Sum of squares};
    \node at (11,-2) {\textcolor{blue}{\textbf{2}}};
     \node at (11,-4.3) {\textcolor{blue}{\textbf{15}}};
     \node at (11,-7) {\textcolor{blue}{\textbf{203}}};
 \node at (14,0) {Bell numbers};
    \node at (14,-2) {\textcolor{blue}{\textbf{B(2)}}};
    \node at (14,-4.3) {\textcolor{blue}{\textbf{B(4)}}};
    \node at (14,-7) {\textcolor{blue}{\textbf{B(6)}}};
\end{tikzpicture}
\label{fig1}
\caption{Induction-restriction diagram 
for the branching $R_{n-1}\subset R_n$, up to level $k= 3$,
the number at the north-east corner of a box is the multiplicity}
\end{figure}

To this end, our first goal is to understand the $R_n$-module
\begin{displaymath}
W_{k,n}:= (\Ind^{R_n}_{R_{n-1}}\circ\Res^{R_n}_{R_{n-1}})^k(\mathbb{C}_{\triv}).
\end{displaymath}
In Theorem~\ref{thm:1}, we give a decomposition of this space 
into tensor powers of the natural $R_n$-module $V_n$ with multiplicities 
invloving various Bell numbers. The dimension of this space itself is 
given by a nice combinatorial entity, namely the generalized Bell number. 
While generalized Bell numbers have been considered in the context of 
number theory and combinatorics, see \cite{1,2,3,4}, our Proposition~\ref{prop:2} 
shows their natural occurrence in representation theory as well. It is interesting 
to note that the module $W_{k,n}$ also helps to give a representation theoretical
interpretation, in  Subsection~\ref{s-comb3}, of a
certain combinatorial identity in~\cite{17} and we also find a new 
identity in Lemma~\ref{lem-s-comb2.1}.

In Section~\ref{sec:cent}, we define an action of the partition algebra 
$\mathcal{P}_k(\delta)$, for any value of the parameter $\delta$, on 
the space $W_{k,n}$, such that this action commutes with the action of 
$R_n$. Further, we prove in Theorem~\ref{thm:main2} that, 
for any parameter $\delta$ such that $\mathcal{P}_k(\delta)$ is semi-simple 
and $n\geq k$, we have an algebra isomorphism 
$$\mathcal{P}_k(\delta)\cong \End_{R_n}(W_{k,n}). $$

To prove the above result we use properties of characters for the 
partition algebra as given by~\cite{6} and Theorem~\ref{thm:cr} 
where we prove that character values for the partition algebras can be 
described in terms of character values for various algebras which 
are isomorphic to the monoid algebras of dual symmetric monoids.  

The paper is organized as follows. In Section~\ref{sec:parti}, we 
collected all necessary preliminaries on partition algebras. 
In Section~\ref{sec:inrook}, we recall the representation theory of 
the rook monoid and the dual symmetric inverse monoid.
In Section~\ref{sec:inrook}, we give a decomposition of $W_{k,n}$ 
and in Section~\ref{sec:cent} we prove that any semi-simple 
partition algebra $\mathcal{P}_k(\delta)$ is the  centralizer of $R_n$
action on $W_{k,n}$, provided that $n\geq k$. The last 
Section~\ref{s-comb} contains various connections 
between the representation $W_{k,n}$ and certain combinatorial 
results from \cite{17}.

\section{Preliminaries}\label{sec:parti}

In this paper, we work over the field $\mathbb{C}$ of complex numbers.

%

\subsection{Partition algebras}\label{sec:parti.1}

Let $X$ be a non-empty set. Recall that a {\em set-partition} of $X$ is a collection 
$\mathcal{S} = \{Y_1, Y_2,\ldots, Y_l\}$ of disjoint non-empty subsets of $X$ 
such that $\displaystyle\bigcup_{i=1}^l Y_i = X$. Each individual
$Y_i$ is called  a {\em part} of $\mathcal{S}$.

For $[k]:=\{1,2,\ldots, k\}$ and $[l']:=\{1',2',\ldots,l'\}$, we are interested 
in set-partitions of the set $[k]\cup [l']$. To start with, let us introduce the following
notation: for a part
\begin{displaymath}
\mathcal{T}= \{a_1,a_2, \ldots, a_{p}\}\cup \{b_1',b_2',\ldots,b_q'\}\subset 
[k]\cup [l'], 
\end{displaymath}
we set
$\mathcal{T}':=\{a_1',a_2',\ldots,a_{p}'\}\cup\{b_1,b_2,\ldots,b_q\}$.

Now we can define the notion of a {\em partition diagram}. For a 
set-partition of $[k]\cup [l']$, the corresponding partition diagram 
is an equivalence class of simple graphs. Each of these graphs has 
$[k]\cup [l']$ as a vertex set. We depict the vertices 
of this graph in two rows as follows: the upper (top) row contains 
$k$ vertices labeled $1, 2,\ldots,k$ from left to right. The lower
(bottom) row contains $l$ vertices labeled $1', 2', \ldots, l'$ 
from left to right. The condition on the edges is that two 
vertices belong to the same connected component if and only
if they belong to the same part of the partition. In particular,
two such graphs represent the same partition diagram if and only if 
they have the same connected components. From now on, we will identify
set-partitions with the corresponding partition diagrams.
We denote the set of partition diagrams on $[k]\cup [k']$ by $\mathcal{A}_k$. 
Abusing terminology, we will call any simple graph 
with vertex set $[k]\cup [l']$ a partition diagram, meaning that we, in fact,
think of the whole partition diagram that this graph represents.

A part of a partition diagram on $[k]\cup [l']$  is {\em propagating} 
provided that it intersects both $[k]$ and $[l']$ non-trivially.
The number of propagating parts of a partition diagram $d$
is called the {\em rank} of $d$ and denoted $\mathrm{rank}(d)$.

\begin{example}\label{ex1}
For $k=7$, $l=4$, the set-partition $\{\{1,3',4'\}, \{2,3,4,1'\},\{2'\}, \{5,6,7\}\}$
can be represented by the following diagram:
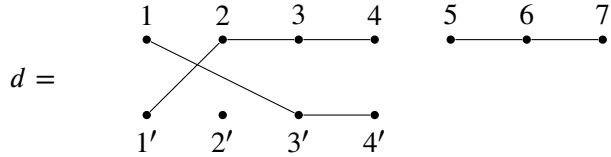
\begin{center}
\begin{tikzpicture}[scale=1,mycirc/.style={circle,fill=black, minimum size=0.1mm, inner sep = 1.1pt}]
\node at (-1.5,0.5) {$d = $};
\node[mycirc,label=above:{$1$}] (n1) at (0,1) {};
\node[mycirc,label=above:{$2$}] (n2) at (1,1) {};
\node[mycirc,label=above:{$3$}] (n3) at (2,1) {};
\node[mycirc,label=above:{$4$}] (n4) at (3,1) {};
\node[mycirc,label=above:{$5$}] (n5) at (4,1) {};
\node[mycirc,label=above:{$6$}] (n6) at (5,1) {};
\node[mycirc,label=above:{$7$}] (n7) at (6,1) {};
\node[mycirc,label=below:{$1'$}] (n1') at (0,0) {};
\node[mycirc,label=below:{$2'$}] (n2') at (1,0) {};
\node[mycirc,label=below:{$3'$}] (n3') at (2,0) {};
\node[mycirc,label=below:{$4'$}] (n4') at (3,0) {};
\path[-, draw](n1) to (n3');
\path[-, draw](n3') to (n4');
\path[-,draw](n2) to (n3) to (n4); 
\path[-,draw](n2) to (n1'); 
\path[-,draw] (n5) to (n6) to (n7);              
\end{tikzpicture}
\end{center}
\end{example}

Let $d$ be a partition diagram of $[k]\cup [l']$. Denote by $\text{top}(d)$
the set-partition of $[k]$ obtained by intersecting the parts of
$d$ with $[k]$ (and removing all empty intersections).
We call $\text{top}(d)$ the {\em top set-partition} of $d$.
For example, if we take $d$ from Example~\ref{ex1}, 
then $\text{top}(d)= \{\{1\}, \{2,3,4\}, \{5,6,7\}\}$. 
Similarly, we denoted by $\text{bottom}(d)$
the set-partition of $[l']$ obtained by intersecting the parts of
$d$ with $[l']$ (and removing all empty intersections).
We call $\text{bottom}(d)$ the {\em bottom set-partition} of $d$.
For example, if we take $d$ from Example~\ref{ex1}, 
then  $\text{bottom}(d)= \{\{1'\}, \{2'\}, \{3',4'\}\}$.

For $k\geq 1$, as a vector space, the {\em partition algebra}
$\mathcal{P}_k(\delta)$  is defined as the linear span of 
all elements in  $\mathcal{A}_k$.

Next let us describe how partition diagrams are multiplied.
The first step in multiplication is vertical concatenation. Given
$d_1, d_2\in\mathcal{A}_k$, their {\em vertical concatenation},
denoted $d_1\circ d_2$, is obtained in the following way:
\begin{itemize}
\item first we place $d_1$ on top of $d_2$;
\item next we identify each vertex in the bottom row of $d_1$
with the corresponding vertex in the top row of $d_2$;
\item the diagram $d_1\circ d_2$ inherits the top row from $d_1$
and the bottom row from $d_2$, and the connectivity of vertices
is determined by the connected components of the previous step.
\end{itemize}

An example of two partition diagrams $d_1$ and $d_2$
from $\mathcal{A}_7$ and their vertical concatenation $d_1\circ d_2$
can be found in Figure~\ref{fig2}. Note that, after the second step
of the above procedure, we will have a graph in which we potentially
have connected components entirely contained in the bottom row of $d_1$
(which is now identified with the top row of $d_2$). Let $c$
be the number of such connected components. In the example given
by Figure~\ref{fig2}, we have two such components highlighted in 
{\color{red}red}.

Now we can define the product of two partition diagrams.
For $d_1,d_2\in\mathcal{A}_k$, the product $d_1d_2$ in
$\mathcal{P}_k(\delta)$ is defined as $\delta^c(d_1\circ d_2)$.
So, in the example given by Figure~\ref{fig2}, we have
$d_1d_2=\delta^2(d_1\circ d_2)$.

\begin{figure}
  \begin{tikzpicture}[scale=1,mycirc/.style={circle,fill=black, minimum size=0.1mm, inner sep = 1.1pt}]
\node at (-1.5,0.5) {$d_1 = $};
\node[mycirc,label=above:{$1$}] (n1) at (0,1) {};
\node[mycirc,label=above:{$2$}] (n2) at (1,1) {};
\node[mycirc,label=above:{$3$}] (n3) at (2,1) {};
\node[mycirc,label=above:{$4$}] (n4) at (3,1) {};
\node[mycirc,label=above:{$5$}] (n5) at (4,1) {};
\node[mycirc,label=above:{$6$}] (n6) at (5,1) {};
\node[mycirc,label=above:{$7$}] (n7) at (6,1) {};

\node[mycirc,label=below:{$1'$}] (n1') at (0,0) {};
\node[mycirc,label=below:{$2'$}] (n2') at (1,0) {};
\node[mycirc,label=below:{$3'$}] (n3') at (2,0) {};
\node[mycirc,label=below:{$4'$}] (n4') at (3,0) {};
\node[mycirc,label=below:{$5'$}] (n5') at (4,0) {};
\node[mycirc,label=below:{$6'$}] (n6') at (5,0) {};
\node[mycirc,label=below:{$7'$}] (n7') at (6,0) {};

\path[-, draw](n1) to (n2);
\path[-, draw](n2) to (n3);
\path[-,draw](n1') to (n2'); 
\path[-,draw](n2') to (n3'); 
\path[-,draw] (n1) to (n1');
\path[-,draw] (n3) to (n3');

\path[-, draw](n4) to (n5);
\path[-,red,draw](n4') to (n5');  
\path[-,draw](n7) to (n7');

\node at (-1.5,-2.5) {$d_2 = $};
\node[mycirc,label=above:{$1$}] (m1) at (0,-2) {};
\node[mycirc,label=above:{$2$}] (m2) at (1,-2) {};
\node[mycirc,label=above:{$3$}] (m3) at (2,-2) {};
\node[mycirc,label=above:{$4$}] (m4) at (3,-2) {};
\node[mycirc,label=above:{$5$}] (m5) at (4,-2) {};
\node[mycirc,label=above:{$6$}] (m6) at (5,-2) {};
\node[mycirc,label=above:{$7$}] (m7) at (6,-2) {};
\path[dotted,draw, thick] (n1')..controls(-0.6,-1)..(m1);
\path[dotted,draw,  thick] (n2')..controls(0.4,-1)..(m2);
\path[dotted,draw, thick] (n3')..controls(1.4,-1)..(m3);
\path[dotted,draw, red,thick] (n4')..controls(2.4,-1)..(m4);
\path[dotted,draw, red,thick] (n5')..controls(3.4,-1)..(m5);
\path[dotted,draw, red, thick] (n6')..controls(4.4,-1)..(m6);
\path[dotted,draw, thick] (n7')..controls(5.4,-1)..(m7);

\node[mycirc,label=below:{$1'$}] (m1') at (0,-3) {};
\node[mycirc,label=below:{$2'$}] (m2') at (1,-3) {};
\node[mycirc,label=below:{$3'$}] (m3') at (2,-3) {};
\node[mycirc,label=below:{$4'$}] (m4') at (3,-3) {};
\node[mycirc,label=below:{$5'$}] (m5') at (4,-3) {};
\node[mycirc,label=below:{$6'$}] (m6') at (5,-3) {};
\node[mycirc,label=below:{$7'$}] (m7') at (6,-3) {};

\path[-,draw] (m6') to (m7');
\path[-,draw](m1) to (m2);
\path[-,draw](m1) to (m1'); 
\path[-,draw](m2') to (m3'); 
\path[-,draw] (m3) to (m3');
\path[-,red,draw] (m4) to (m5);
\path[-,draw](m7) to (m7');                
\end{tikzpicture}

\vspace{10mm}

\begin{tikzpicture}[scale=1,mycirc/.style={circle,fill=black, minimum size=0.1mm, inner sep = 1.1pt}]
\node at (-1.5,0.5) {$d_1\circ d_2 = $};
\node[mycirc,label=above:{$1$}] (n1) at (0,1) {};
\node[mycirc,label=above:{$2$}] (n2) at (1,1) {};
\node[mycirc,label=above:{$3$}] (n3) at (2,1) {};
\node[mycirc,label=above:{$4$}] (n4) at (3,1) {};
\node[mycirc,label=above:{$5$}] (n5) at (4,1) {};
\node[mycirc,label=above:{$6$}] (n6) at (5,1) {};
\node[mycirc,label=above:{$7$}] (n7) at (6,1) {};

\node[mycirc,label=below:{$1'$}] (n1') at (0,0) {};
\node[mycirc,label=below:{$2'$}] (n2') at (1,0) {};
\node[mycirc,label=below:{$3'$}] (n3') at (2,0) {};
\node[mycirc,label=below:{$4'$}] (n4') at (3,0) {};
\node[mycirc,label=below:{$5'$}] (n5') at (4,0) {};
\node[mycirc,label=below:{$6'$}] (n6') at (5,0) {};
\node[mycirc,label=below:{$7'$}] (n7') at (6,0) {};
\path[-,draw] (n1) to (n2) to (n3);
\path[-, draw] (n1') to (n2') to (n3');
\path[-, draw] (n1) to (n1');
\path[-, draw] (n4) to (n5);
\path[-,draw] (n6') to (n7');
\path[-,draw] (n7) to (n7');
\end{tikzpicture}
\label{fig2}
\caption{Example of vertical concatenation of partition diagrams}
\end{figure}
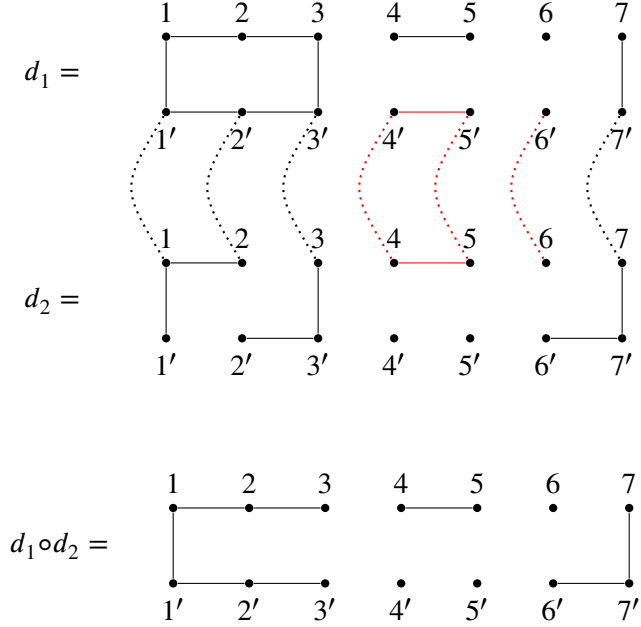   

Partition diagrams can also be concatenated horizontally
(which represents the fact that all partition diagrams 
of all possible $[k]\cup[l']$ can be organized into a $2$-category).
Given $d_1\in\mathcal{A}_{k_1}$ and $d_2\in\mathcal{A}_{k_2}$, their
{\em horizontal concatenation} $d_1\otimes d_2\in\mathcal{A}_{k_1+k_2}$
is obtained by drawing $d_2$ next to $d_1$ on the right and 
renumbering the vertices of $d_2$ appropriately. 

For instance, the horizontal concatenation of the diagrams 
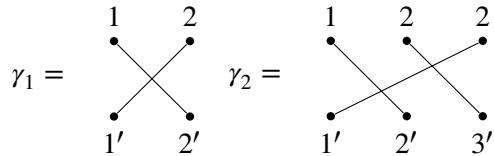
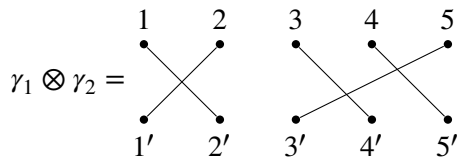
\begin{center}
    \begin{tikzpicture}[scale=1,mycirc/.style={circle,fill=black, minimum size=0.1mm, inner sep = 1.1pt}]
\node at (-1,0.5) {$\gamma_1 = $};
\node[mycirc,label=above:{$1$}] (n1) at (0,1) {};
   \node[mycirc,label=above:{$2$}] (n2) at (1,1) {};

   \node[mycirc,label=below:{$1'$}] (n1') at (0,0) {};
				\node[mycirc,label=below:{$2'$}] (n2') at (1,0) {};
				
          \path[-, draw](n1) to (n2');
				\path[-,draw](n1') to (n2);  
               
	\end{tikzpicture}
  \begin{tikzpicture}[scale=1,mycirc/.style={circle,fill=black, minimum size=0.1mm, inner sep = 1.1pt}]
\node at (-1,0.5) {$\gamma_2 = $};
\node[mycirc,label=above:{$1$}] (n1) at (0,1) {};
\node[mycirc,label=above:{$2$}] (n2) at (1,1) {};
\node[mycirc,label=above:{$2$}] (n3) at (2,1) {};

\node[mycirc,label=below:{$1'$}] (n1') at (0,0) {};
\node[mycirc,label=below:{$2'$}] (n2') at (1,0) {};
\node[mycirc,label=below:{$3'$}] (n3') at (2,0) {};
				
\path[-, draw](n1) to (n2');
\path[-,draw](n2) to (n3');  
 \path[-,draw](n3) to (n1');                
	\end{tikzpicture}
 \end{center} 
 equals
 \begin{center}
  \begin{tikzpicture}[scale=1,mycirc/.style={circle,fill=black, minimum size=0.1mm, inner sep = 1.1pt}]
\node at (-1,0.5) {$\gamma_1\otimes \gamma_2 = $};
\node[mycirc,label=above:{$1$}] (n1) at (0,1) {};
\node[mycirc,label=above:{$2$}] (n2) at (1,1) {};
\node[mycirc,label=above:{$3$}] (n3) at (2,1) {};
\node[mycirc,label=above:{$4$}] (n4) at (3,1) {};
\node[mycirc,label=above:{$5$}] (n5) at (4,1) {};

\node[mycirc,label=below:{$1'$}] (n1') at (0,0) {};
\node[mycirc,label=below:{$2'$}] (n2') at (1,0) {};
\node[mycirc,label=below:{$3'$}] (n3') at (2,0) {};
\node[mycirc,label=below:{$4'$}] (n4') at (3,0) {};
\node[mycirc,label=below:{$5'$}] (n5') at (4,0) {};

\path[-, draw](n1) to (n2');
\path[-,draw](n1') to (n2);  
\path[-, draw](n3) to (n4');
\path[-,draw](n4) to (n5');  
 \path[-,draw](n5) to (n3');                
	\end{tikzpicture}
 \end{center} 
 
The way we draw and multiply diagrams suggests that we should 
naturally consider top and bottom modules. Our previous 
notation reflects the standard  convention 
that the top structure corresponds to left modules while the
bottom structure to right modules.
 
\subsection{Symmetric groups} 

In this article, symmetric groups will appear in various context, so here we
fix some notation and recall some standard facts, see e.g. \cite{Sa}.

For a positive integer $n$, let $S_n$ denote the {\em symmetric group} on $[n]$. 
The irreducible representations of $S_n$ are indexed by partitions of $n$. 
For a partition $\lambda$ of $n$ (denoted $\lambda\vdash n$), let 
$\text{SYT}(\lambda)$ denote the set of {\em  standard Young tableaux} 
of shape $\lambda$. We also denote by $S^\lambda$ the {\em Specht module} 
corresponding to $\lambda$. The module $S^\lambda$ has a basis
$\{v_T\mid T\in\text{SYT}(\lambda)\}$ indexed by $\text{SYT}(\lambda)$
and consisting of so-called {\em polytabloids}. The dimension of 
$S^\lambda$, that is, the cardinality of $\text{SYT}(\lambda)$,
is denoted $f^\lambda$.

\subsection{Schur--Weyl duality between symmetric group and partition algebra}\label{sec:phi} 

For $n\geq 1$, consider the defining representation $\mathbb{C}^n$ of $S_n$ 
and let $e_1, e_2,\ldots, e_n$ be the standard basis of $\mathbb{C}^n$. 
For $k\geq 1$, the partition algebra $\mathcal{P}_k(n)$ acts on (the right of) 
$(\mathbb{C}^n)^{\otimes k}$ as follows.

For $d\in\mathcal{A}_k$ and a basis vector $e_{i_1}\otimes e_{i_2}\otimes\cdots\otimes e_{i_k}\in(\mathbb{C}^n)^{\otimes k}$,
$$\phi_k(d)(e_{i_1}\otimes e_{i_2}\otimes\cdots\otimes e_{i_k}):= 
\sum e_{j_1}\otimes e_{j_2}\otimes \cdots\otimes e_{j_k}, $$
where the sum is taken over all $j_1,j_2,\dots, j_k$ which satisfy 
the following conditions:
\begin{itemize}
\item if some $p$ and $q$ appear in the same part of $d$,
then $i_p=i_q$;
\item if some $p'$ and $q'$ appear in the same part of $d$,
then $j_p=j_q$;
\item if some $p$ and $q'$ appear in the same part of $d$,
then $i_p=j_q$
\end{itemize}
(note that the first of these conditions does not depend on 
$j_1,j_2,\dots, j_k$ at all). Furthermore, for all $\sigma\in S_n$, we have $
(\sigma \circ \phi_{k})(d) = (\phi_k\circ \sigma)(d)$, so 
$\phi_k(d)\in\text{End}_{S_n}((\mathbb{C}^n)^{\otimes k})$.
The following result is due to Jones and Martin, see \cite{Jones,Martin}.

\begin{theorem}
For $k\geq 1$, the map 
$\phi_k:\mathcal{P}_k(n)\to \textrm{End}_{S_n}((\mathbb{C}^n)^{\otimes k})$ 
is surjective. Moreover, for $n\geq 2k$, the map $\phi_k$ is injective. 
\end{theorem}

\subsection{Irreducible representations and characters of partition algebras}

In this section, we recall the basic representation theory of  $\mathcal{P}_k(\delta)$,
see e.g \cite{13}. 

Let $V(k,i)$ denote the set of all partition diagrams in $\mathcal{A}_k$
that have rank $i$ for which all elements of $[k']$
belong to different parts and, moreover, the elements $1', 2', \ldots, i'$
belong to propagating parts. Then the linear span 
$\mathbb{C}[V(k,i)]$ becomes a left module over $\mathcal{P}_k(\delta)$ 
as follows, for $d_1\in V(k,i)$ and $d\in \mathcal{A}_k$ we have
\begin{displaymath}
d \cdot d_1:= \begin{cases}
    dd_1, & \text{ if } d\circ d_1\in V(k,i);\\
    0,  & \text{ otherwise. }
\end{cases} 
\end{displaymath}
The symmetric group $S_i$ acts on the right of $\mathbb{C}[V(k,i)]$
naturally, permuting the elements $1',2',\dots,i'$. This action obviously
commutes with the left action of $\mathcal{P}_k(\delta)$.
For $\lambda\vdash i$, the corresponding {\em standard module}
$\mathcal{P}_k^\lambda$ is defined as
\begin{displaymath}
\mathcal{P}_k^\lambda:= \mathbb{C}[V(k,i)]\otimes_{S_i} S^\lambda.
\end{displaymath}

For $\delta\in \mathbb{C}$ such that  $\mathcal{P}_k(\delta)$ is semi-simple,
the set $\{\mathcal{P}_k^\lambda\mid 0\leq i\leq k,\, \lambda\vdash i\}$ is 
a complete and irredundant set of simple $\mathcal{P}_k(\delta)$-modules.

\subsection{Set-partition tableaux}  

Let $A_1$ and $A_2$ be two non-empty disjoint subsets of the set
$\mathbb{Z}_{>0}$. We write $A_1\prec A_2$ provided that $\min A_1< \min A_2.$ 
This is usually called the {\em minimum entry order}. 

Let $\mu$ be some Young diagram with $k$ boxes and $A_1,A_2,\dots,A_k$
be a collection of non-empty disjoint subsets of $\mathbb{Z}_{>0}$.
A {\em set-partition tableau} of shape $\mu$ is a bijective 
assignment of some $A_i$ to each box of $\mu$. The union of all
the $A_i$'s is called the {\em content} of our set-partition tableau.
A set-partition 
tableau is called {\em standard} if its entries increase 
left-to-right and top-to-bottom with respect to the minimum entry order.
For $0\leq|\mu|\leq k$, let $\text{Set-SYT}_k(\mu)$ denote the set of all 
standard set-partition tableaux of shape $\lambda$ 
and content $[k]$.

Let $\mathcal{N}(k,i)$ be the set of all those elements of $V(k,i)$ whose 
top blocks, say $C_1, C_2,\ldots, C_i$, connected to $1',2',\ldots, i'$,
respectively, are ordered with respect to the minimum entry order, i.e.,
$$ C_1\prec C_2\prec\cdots\prec C_i.$$
The set $\mathcal{N}(k,i)$ is a cross-section of the orbits of the 
natural action of 
$S_i$ on $V(k,i)$. Consequently, the set 
$$\{d\otimes v_T\mid d\in\mathcal{N}(k,i), T\in\text{SYT}(\lambda)\}$$
is a basis of $\mathcal{P}_k^\lambda$.

For $T\in\text{SYT}(\lambda)$, we can assign $C_1$ to the block of $T$
containing $1$, then we assign $C_2$ to the block of $T$
containing $2$, and so on. This produces a standard set-partition 
tableau $P_T$ of shape $\lambda$. For the same $d\in\mathcal{N}(k,i)$, 
let $Q$ be the standard set-partition tableau with one row  
whose filling is given by the non-propagating top parts of $d$.  Then 
the set
$$\{(P_T,Q)\mid T\in\text{SYT}(\lambda), P_T\leftrightarrow d\in\mathcal{N}(k,i)\} $$
naturally indexes a basis of $\mathcal{P}_k^\lambda$.

\subsection{Characters of partition algebras}\label{sec:char}

Characters of partition algebras are studied in \cite{6} in detail.
The definition is taken from representation theory of finite groups:
given a finite dimensional $\mathcal{P}_k(\delta)$ module $V$, its 
{\em character} $\chi_{V}$ assigns to $u\in \mathcal{P}_k(\delta)$
the trace of the linear operator with which $u$ acts on $V$.

For the cycle $(1, 2, \ldots, i)$ in $S_i$, 
let $\gamma_i$ denote the corresponding partition diagram 
in $\mathcal{A}_i$ and $E_{1} = \{\{1\}, \{1'\}\}\in\mathcal{A}_1$.
For a composition $\mu = (\mu_1, \mu_2, \ldots,\mu_l)$ of $m \leq k$, 
consider the following partition diagram:
$$d_\mu = \gamma_{\mu_1}\otimes \gamma_{\mu_2}\otimes\cdots
\otimes \gamma_{\mu_l}\otimes E_1^{\otimes (k-m)}.$$

These elements $d_\mu$ completely determine the character of
any $\mathcal{P}_k(\delta)$-module in the following sense. Given $d\in \mathcal{A}_k$,
the rank of 
\begin{displaymath}
d(i):=\underbrace{d\circ d\circ \dots\circ d}_{i\text{ factors}} 
\end{displaymath}
weakly decreases with $i$ and hence eventually
stabilizes for $i\gg 0$. Let $m$ be the corresponding minimal
value of this rank. There is an $i$ such that the rank of
$d(i)$ equals $m$ and for which one can find two partition
diagrams $d_1$ and $d_2$ of rank $m$ such that 
$d_1\circ d_2= d_\mu$, for some $\mu$ as above, and
$d_2\circ d_1= d(i)$. This implies that
\begin{equation}\label{eq:inde}
\chi_{V}(d) = \delta^s \chi_{V}(d_\mu),
\end{equation}
for any $\mathcal{P}_k(\delta)$-module $V$,
where $s$ is a nonnegative integer which depends only on $d$
but not on $V$.

We refer to \cite[Section~3]{6} for details
and also to \cite{KM09b} which explains similar
phenomena from the point of view of representations of semigroups.

\section{Dual symmetric inverse monoid and rook monoid}\label{sec:inrook}

In this section we recall basic representation theory of two monoids,
the dual symmetric inverse monoid and the rook monoid. Both these
monoids are inverse monoids. The complex representation theory of 
an inverse monoid is semi-simple and is built from the representation
theory of its maximal subgroups. The latter are in bijection with
idempotents, however, for representation theory, the essential 
information is provided by equivalence classes of such idempotents
where the equivalence is given by Green's $\mathcal{J}$-relation.
We use \cite{8} as a general reference for representation theory of 
monoids, in particular, inverse monoids.

\subsection{Dual symmetric inverse monoid}\label{sec:dual}

Denote by $\mathcal{I}_k^*$ the set of all partition diagrams in 
$\mathcal{A}_k$ for which each part is propagating.
It turns out that the product of two elements in 
$\mathcal{I}_k^*$ coincides with vertical concatenation
(so, no additional $\delta$'s appear). In particular,
$\mathcal{I}_k^*$ is a multiplicative submonoid of 
$\mathcal{P}_k(\delta)$, for any $\delta\in\mathbb{C}$.
The monoid algebra $\mathbb{C}[\mathcal{I}_k^*]$ is therefore
a subalgebra of the partition algebra $\mathcal{P}_k(\delta)$,
for any $\delta\in\mathbb{C}$. As $\mathcal{I}_k^*$ is an inverse 
monoid, the monoid algebra $\mathbb{C}[\mathcal{I}_k^*]$ is semi-simple. 

\subsection{Representation theory of $\mathcal{I}_k^*$}\label{dualreps} 

The irreducible representations of $\mathbb{C}[\mathcal{I}_k^*]$ 
are indexed by partitions $\lambda$ of $i$, where 
$1\leq i\leq k$. For $1\leq i\leq k$, 
let $\epsilon_i$ denote the partition diagram 
whose parts are $\{1,1'\}$, $\{2,2'\}, \ldots, \{i-1,(i-1)'\}$ and
\begin{equation}\label{eq:sa}
\{i, i+1, \ldots, k, i', (i+1)',\ldots, k'\}.
\end{equation}
Let $G_i$ denote the maximal subgroup corresponding to $\epsilon_i$. 
Then $G_i$ consists of all partitions diagrams in $\mathcal{I}_k^*$ 
of rank $i$ with the corresponding top set partitions
given by
\begin{displaymath}
\{\{1\},\{2\},\dots,\{i-1\},\{i,i+1,\dots,k\}\}. 
\end{displaymath}
and the corresponding bottom set partitions given by
\begin{equation}\label{eqeq1}
\{\{1'\},\{2'\},\dots,\{(i-1)'\},\{i',(i+1)',\dots,k'\}\}. 
\end{equation}
The subgroup $G_i$ is isomorphic to the symmetric group $S_{i}$.

Let $\mathcal{L}(k,i)$ denote Green's left class containing 
the idempotent $\epsilon_i$. It is easy to check that 
$\mathcal{L}(k,i)$ consists of partition diagrams of rank $i$
in $\mathcal{I}_k^*$ whose bottom set-partition is given by
\eqref{eqeq1}.
The linear span $\mathbb{C}[\mathcal{L}(k,i)]$ is a 
$\mathbb{C}[\mathcal{I}_k^*]$-$\mathbb{C}[G_i]$-bimodule
in the obvious way. 
For $\lambda\vdash i$ and the Specht $S_i$-module $S^\lambda$, 
the corresponding irreducible representation of 
$\mathbb{C}[\mathcal{I}_k^*]$ is given by
\begin{equation}
\mathcal{I}_k^\lambda:= \mathbb{C}[\mathcal{L}(k,i)]\otimes_{G_i} S^\lambda. 
\end{equation}
The set $\{\mathcal{I}_k^\lambda\mid \lambda\vdash i, 1\leq i\leq n \}$
is a complete and irredundant list of representatives of the isomorphism
classes of simple $\mathbb{C}[\mathcal{I}_k^*]$-modules.

Let $\mathcal{K}(k,i)$ consist of those diagrams in 
$\mathcal{L}(k,i)$ whose top-set partitions $C_1, C_2, \ldots, C_i$, 
that are connected to $\{1'\}, \{2'\},\ldots,\{(i-1)'\}, \{i',(i+1)',\dots, k'\}$, 
respectively, are ordered with respect to the minimum entry order, i.e.,
$\min C_1 <\min C_2<\cdots<\min C_i$.
Then 
\begin{displaymath}
\mathcal{B}(k,i):=\{d\otimes v_T\mid d\in \mathcal{K}(k,i) \text{ and } T\in\text{SYT}(\lambda)\} 
\end{displaymath}
is a basis of $\mathcal{I}_k^\lambda$.

\subsection{Rook monoids}

Let $R_n$ denote the set consisting of all partial permutation matrices of 
size $n$. This means that each entry of the matrix is either $0$ or $1$
and each row and each column contains at most one entry equal to $1$
(these correspond to the so-called non-attacking rook configurations).
Then $R_n$ is closed under matrix multiplication and hence is a monoid,
called the {\em rook monoid}, alternatively the {\em  symmetric inverse monoid}. 

\subsection{Representation theory of $R_n$}

In this subsection, we recall irreducible representations of 
$R_n$ over $\mathbb{C}$. Again, for details we refer to 
\cite{8} or \cite{GM}. 

For $i=0,1,2,\dots,n$, let $e_i$ be the diagonal matrix in which the
first $i$ diagonal elements are $1$ and all the remaining elements are $0$.
These form a cross-section of Green's $\mathcal{J}$-classes in $R_n$.
The maximal subgroup corresponding to the idempotent $e_i$ is isomorphic to the symmetric group $S_i$. Let $\mathbb{L}_i$ denote the Green's left class 
containing $e_i$.

As usual, the linear span $\mathbb{C}[\mathbb{L}_i]$ is a 
($R_n$-$S_i$)-bimodule. For $0\leq i\leq n$ and $\lambda\vdash i$, 
the $R_n$-module
\begin{equation}
R_n^\lambda = \mathbb{C}[\mathbb{L}_i]\otimes_{S_i} S^\lambda
\end{equation}
is a simple $R_n$-module. In particular, for $i=0$ the corresponding
module $ R_n^\varnothing$ is the {\em trivial module}
$\mathbb{C}_{\text{triv}}$ of $R_n$.
For $i=1$, the module $ R_n^{(1)}$ is the {\em natural} 
(a.k.a. {\em defining}) module 
$\mathbb{C}^n$, on which elements of $R_n$ act simply by 
matrix multiplication. The set 
$\{R_n^\lambda\mid \lambda\vdash i, 0\leq i\leq n \}$
is a complete and irredundant list of representatives of the isomorphism
classes of simple $R_n$-modules.

\subsection{Induction and restriction}

For $n\geq 1$, the rook monoid $R_{n-1}$ embeds into $R_{n}$ 
in the obvious way, that is, by sending a matrix $A\in R_{n-1}$
to the matrix $\left(\begin{array}{cc}A&0\\0&1\end{array}\right)$.

As usual, we denote by $R_n\text{-mod}$ the category of 
finitely generated (and hence finite dimensional) 
$R_n$-modules. With respect to the above monoid embedding of 
$R_{n-1}$ into $R_{n}$, we have the corresponding 
induction and restriction functors:
\begin{equation}
\text{Ind}^{R_n}_{R_{n-1}}: R_{n-1}\text{-mod}\to R_n\text{-mod} \hspace{0.5cm} \text{ and } \hspace{0.5cm} \text{Res}^{R_{n}}_{R_{n-1}}: R_{n}\text{-mod}\to R_{n-1}\text{-mod} .
\end{equation}

Recall, from \cite[Corollary~3.3]{10}, the following description of the 
restriction (resp. the induction) of an irreducible 
representation of $R_n$ (resp. $R_{n-1}$) to $R_{n-1}$ (resp. $R_n$) 
(see also~\cite{9} for a characteristic-free description).
We identify partitions and Young diagrams.
Given a partition $\lambda$, let $\lambda^{+}$ denote the set of 
all partitions obtained from $\lambda$ by adding an addable node.
Similarly, we denote by $\lambda^{-}$ the set of all partitions
obtained from  $\lambda$ by removing a removable node.
For $0\leq i\leq (n-1)$ and $\lambda\vdash i$, we have 
\begin{equation}
\text{Ind}^{R_n}_{R_{n-1}}(R_{n-1}^\lambda)\cong R_{n}^\lambda\oplus \bigoplus_{\mu\in\lambda^+}R_{n}^{\mu},
\end{equation}
From this, we see that $\text{Ind}^{R_n}_{R_{n-1}}$ can be written as
$$ \Ind^{R_n}_{R_{n-1}} 
\cong \widehat{\Ind}^{R_n}_{R_{n-1}}\oplus F, \hspace{0.5cm} $$
where
$\widehat{\Ind}^{R_n}_{R_{n-1}}: R_{n-1}\text{-mod} \to R_{n}\text{-mod}$
and $\mathrm{G}: R_{n-1}\text{-mod} \to R_{n}\text{-mod}$ are given by
\begin{displaymath}
\widehat{\Ind}^{R_n}_{R_{n-1}}(R^{\lambda}_{n-1}) = 
 \bigoplus_{\mu\in\lambda^{+}} R_n^{\mu}
\quad \hspace{0.3cm} \text{ and } \hspace{0.3cm}\quad 
\mathrm{G}(R^{\lambda}_{n-1}) = R^{\lambda}_{n}.
\end{displaymath}

For $0\leq j\leq n$ and $\mu\vdash j$, we similarly have
\begin{equation}\label{eq:res}
\text{Res}^{R_n}_{R_{n-1}}(R_{n}^\mu)\cong \begin{cases}
\displaystyle
R_{n-1}^\mu\oplus \bigoplus_{\nu\in\mu^-}R_{n-1}^{\nu}, & \text{if } |\mu|<n;\\
\\ \displaystyle
\bigoplus_{\nu\in\mu^-}R_{n-1}^{\nu}, & \text{if } |\mu|=n.
\end{cases}
\end{equation}

\subsection{Various identities}

In this section, we derive some identities involving the 
functors $\widehat{\Ind}^{R_n}_{R_{n-1}}$, $\mathrm{G}$ and 
$\Res^{R_n}_{R_{n-1}}$ from the previous subsection.
We will need these identities in the next section. 

For $M\in R_n\text{-mod}$,  from~\cite[Proposition~2.22]{11}, 
we have that the functor $\widehat{\Ind}^{R_n}_{R_{n-1}}$ 
satisfies the following:
\begin{equation}\label{eq:11}
(\widehat{\Ind}^{R_n}_{R_{n-1}}\circ \Res^{R_n}_{R_{n-1}})(M)\cong M\otimes \C^n.
\end{equation}
For $M_1, M_2\in R_n\text{-mod}$, we, obviously, have  
\begin{equation}\label{eq:12}
\Res^{R_n}_{R_{n-1}}(M_1\otimes M_2) = 
\Res^{R_n}_{R_{n-1}}(M_1)\otimes \Res^{R_n}_{R_{n-1}}(M_2).
\end{equation}

\begin{proposition}\label{prop-identity3}
For $i\in\mathbb{Z}_{>0}$, we have 
\begin{equation}\label{eq:13}
(\mathrm{G}\circ \Res^{R_n}_{R_{n-1}})((\C^n)^{\otimes i})\cong 
(\C^n\oplus \C_{\rm{triv}})^{\otimes i}.
\end{equation} 
\end{proposition}

\begin{proof}
Using~\eqref{eq:res} and \eqref{eq:12}, we rewrite
\eqref{eq:13} as
\begin{displaymath}
\mathrm{G} ((\C^{n-1}\oplus \C_{\text{triv}})^{\otimes i})\cong 
(\C^n\oplus \C_{\text{triv}})^{\otimes i}.
\end{displaymath}
Using the additivity of the tensor product, to prove the above,
it is enough to show that 
\begin{displaymath}
\mathrm{G} ((\C^{n-1})^{\otimes i})\cong 
(\C^n)^{\otimes i},
\end{displaymath}
for every $i$. The latter follows directly from the definition
of $\mathrm{G}$ and the observation in \cite[Example~3.18]{Solomon} 
or \cite[Theorem~2.20]{11} that, for a fixed 
$\lambda\vdash k\leq n$, the multiplicity of the 
simple module $R_n^\lambda$ in  
$(\C^n)^{\otimes i}$ depends on $\lambda$
and $i$ but not on $n$.
\end{proof}

The proof of Proposition~\ref{prop-identity3} seems to suggest that
$\mathrm{G}$ is a monoidal function. We will not need this property
and, at the moment, we do not know how to prove it, if it is true.

\subsection{Schur--Weyl duality between rook monoid and dual symmetric inverse monoid}

The Schur--Weyl duality between the rook monoid $R_n$ and the 
dual symmetric inverse monoid both acting on the tensor space 
$(\mathbb{C}^n)^{\otimes k}$ was discovered in~\cite{5}. 
Here the dual symmetric inverse monoid $\mathcal{I}_k^*$ 
acts on $(\mathbb{C}^n)^{\otimes k}$ by  restriction
from the partition algebra. This Schur--Weyl duality 
(and a mild extension of it, see Proposition~\ref{prop:mild}) 
is crucial in the proof of Theorem~\ref{thm:cr}. 

\begin{theorem}
Consider the obvious action of $R_n$ on $(\mathbb{C}^n)^{\otimes k}$.
Then the image of $\mathbb{C}[R_n]$ in
$\text{End}_{\mathbb{C}}((\mathbb{C}^n)^{\otimes k})$
under this action coincides with 
$\text{End}_{\mathbb{C}[\mathcal{I}_k^*]}((\mathbb{C}^n)^{\otimes k})$. 
\end{theorem}

An immediate corollary of the above theorem combined with the usual 
double centralizer property for semi-simple algebra, see e.g. \cite{13},
we have the following.

\begin{corollary}
We have the following decomposition of $(\mathbb{C}^n)^{\otimes k}$,
as a $(\mathbb{C}[R_n]$-$\mathbb{C}[\mathcal{I}_k^*])$-bimodule:
$$ (\mathbb{C}^n)^{\otimes k} \cong \bigoplus_{\substack{\lambda\vdash i\\ 1\leq i\leq \min\{n,k\}}} R_n^\lambda\otimes \mathcal{I}_k^\lambda.$$
\end{corollary}
    
\section{Main results}

\subsection{Iterated restriction-induction}\label{subsec}

In this section, we study the iterated restriction-induction of 
the trivial representation of the rook monoid $R_n$.
As we already mentioned in the introduction,
in the case of the symmetric group $S_n$,
the iterated restriction-induction of 
the trivial representation is isomorphic to a tensor power 
of the defining representation of $S_n$ which is exactly
what appears in the corresponding Schur-Weyl duality.
Its connection to the construction of the Bratteli diagram 
for a multiplicity-free tower of partition algebras was 
elaborated in \cite[Section~3]{13}.

For $k\geq 1$ and fixed $n\geq 1$, consider the following $k$-fold
iteration of the restriction-induction of the trivial representation 
$ \C_{\text{triv}}$ of $R_{n}$:
\begin{displaymath}
W_{k,n}= \left(\Ind^{R_{n}}_{R_{n-1}}\circ\Res^{R_{n}}_{R_{n-1}}\right)^{k} (\C_\text{triv}). 
\end{displaymath}
The following theorem is one of our main results.  It describes a 
decomposition of $W_{k,n}$ into various tensor powers of the 
defining representation $\mathbb{C}^n$. It also connects
a combinatorial entity to counting certain multiplicities in $W_{k,n}$. 
Ultimately, we want to show that the algebra $\mathcal{P}_k(\delta)$
acts on $W_{k,n}$ centralizing the action of $R_n$. 
The following theorem is an important step in that direction.
 
\begin{theorem}\label{thm:1}
Let $V_n=\mathbb{C}^n$. For $k\in\mathbb{Z}_{>0}$, we have the 
following isomorphism of $R_n$-modules:
$$ W_{k,n} \cong m_{k}^{k} V_n^{\otimes k} \oplus m^{k}_{k-1} V_n^{\otimes k-1}\oplus \cdots\oplus m^{k}_{k-r}V_n^{\otimes k-r}\oplus \cdots \oplus m^{k}_{0} V_n^{\otimes 0}, $$
where $m^{k}_i =\binom{k}{i}B(k-i)$, for $ 0\leq i\leq k$.
\end{theorem}

\begin{proof}
We prove the statement using induction on $k$. For $k=1$, we have:
\begin{displaymath}
W_{1,n} = \Ind^{R_n}_{R_{n-1}}\circ \Res^{R_n}_{R_{n-1}}(\C_{\triv}) 
\cong \Ind^{R_n}_{R_{n-1}}(\C_{\triv})\cong V_n\oplus \C_{\triv}.
\end{displaymath}
For $k=2$, we have:
\begin{align*}
W_{2,n} & = \left(\Ind^{R_n}_{R_{n-1}}\circ \Res^{R_n}_{R_{n-1}}\right)^{2}(\C_{\triv})\\ 
&\cong \left(\Ind^{R_n}_{R_{n-1}}\circ \Res^{R_n}_{R_{n-1}}\right) (V_n\oplus \C_{\triv})\\
&\cong \left(\widehat{\Ind}^{R_n}_{R_{n-1}}\circ \Res^{R_n}_{R_{n-1}}\right) (V_n\oplus \C_{\triv})
\oplus \left(\mathrm{G}\circ \Res^{R_n}_{R_{n-1}}\right)(V_n \oplus \C_{\triv})\\
&\cong V_n^{\otimes 2}\oplus 2 V_n\oplus 2\C_{\triv}.
\end{align*}
For $0\leq i\leq k$,  let $m_{i}^k$ denote the multiplicity of 
$V_n^{\otimes i}$ in  $W_{k,n}$. Then, from the above, for $k=1, 2$, 
these multiplicities are as follows: 
$$m^{1}_{i} = \binom{1}{i}B(1-i)\quad \text{ and }\quad
m^{2}_i=\binom{2}{i}B(2-i).$$

To proceed, assume that
$$W_{k-1,n} = \left(\Ind^{R_n}_{R_{n-1}}\circ \Res^{R_{n}}_{R_{n-1}}\right)^{(k-1)}\cong m_{k-1}^{k-1}(V_n)^{\otimes (k-1)}\oplus m_{k-2}^{k-1}(V_n)^{\otimes (k-2)}\oplus \cdots \oplus m_{1}^{k-1}V_n\oplus m_{0}^{k-1}\C_{\triv},$$
where 
$$m_{i}^{k-1}= \binom{k-1}{i}B(k-1-i), \, \,\text{ for }\,\, 0\leq i\leq (k-1).$$

Using~\eqref{eq:11}-\eqref{eq:13} and the induction hypothesis, we compute:
\begin{align*}
    W_{k,n}&\cong \left(\widehat{\Ind}^{R_n}_{R_{n-1}}\circ \Res^{R_{n}}_{R_{n-1}}\right)\circ \left(\Ind^{R_n}_{R_{n-1}}\circ \Res^{R_{n}}_{R_{n-1}}\right)^{(k-1)}(\C_{\triv})\oplus\\
  &\hspace{0.5cm} \oplus \left(\mathrm{G}\circ \Res^{R_n}_{R_{n-1}}\right)\circ \left(\Ind^{R_n}_{R_{n-1}}\circ \Res^{R_{n}}_{R_{n-1}}\right)^{(k-1)}(\C_{\triv})\\
& \cong V_n\otimes \left(m_{k-1}^{k-1}V_n^{\otimes (k-1)}\oplus m_{k-2}^{k-1}V_n^{\otimes (k-2)}\oplus \cdots \oplus m_{1}^{k-1}V_n\oplus m_{0}^{k-1}\C_{\triv}\right)\oplus\\
& \oplus m_{k-1}^{k-1}(V_n\oplus \C_{\triv})^{\otimes (k-1)}\oplus m_{k-2}^{k-1}(V_n\oplus \C_{\triv})^{\otimes (k-2)}\oplus \cdots \oplus m_{1}^{k-1}(V_n\oplus\C_{\triv})\oplus m_{0}^{k-1}\C_{\triv}.
\end{align*}
Then, for $0\leq i\leq k$, the multiplicity $m^k_{k-i}$ of 
$V_n^{\otimes(k-i)}$ in $W_{k,n}$ con be computed as  follows:
\begin{align*} m^{k}_{k-i} &= m_{k-i-1}^{k-1}+
\binom{k-i}{0}m^{k-1}_{k-i}+
\binom{k-i-1}{1}m^{k-1}_{k-i+1}+
\cdots+\binom{k-2}{i-2}m^{k-1}_{k-2}+\binom{k-1}{i-1}m^{k-1}_{k-1}\\
&=\binom{k-1}{i}B(i)+\binom{k-1}{i-1}B(i-1)+\binom{k-i+1}{1}
\binom{k-1}{i-2}B(i-2)+\\&\,\,\,\,\,+\binom{k-i+2}{2}
\binom{k-1}{i-3}B(i-3)+\cdots+ \binom{k-2}{i-2}\binom{k-1}{i}B(1)+
\binom{k-1}{i-1}\binom{k}{0}B(0).
\end{align*}
Using the identity $$\binom{k-i+m}{m}\binom{k-1}{i-m+1}=\binom{k-1}{i-1}\binom{i-1}{i-m-1},$$  together with Pascal's identity for binomial coefficients and 
the fact that Bell numbers satisfy the recurrence
$B(i)=\displaystyle{\sum_{l=0}^{i-1}}\binom{i-1}{l}B(l)$, we obtain that
$m^{k}_{k-i} = \displaystyle{\binom{k}{k-i}}B(i)$ and we are done.
\end{proof}

\subsection{Restricted set-partitions and generalized Bell numbers}

In what follows, $e_1, \ldots, e_n$ denotes the standard basis vectors of $\C^n$. Let $e_0$ denote a basis of the trivial representation $\C_{\text{triv}} $ so that
$$ \sigma e_{0} = e_0, \, \, \forall \sigma\in R_n. $$
From the decomposition in Theorem~\ref{thm:1}, it follows that 
the following vectors form a basis in $W_{k,n}$:
\begin{equation}\label{eq:1}
(v_{i_1}\otimes v_{i_2}\otimes \cdots \otimes v_{i_k},  \mathcal{D}),
\end{equation}
where $\mathcal{D} =\{C_1,C_2,\ldots, C_l\}$ is a set-partition of a subset of $[k]$ such that $v_{i_j}= e_0$ if and only if  $j\in C_1\cup C_2\cup\cdots\cup C_l$.
In particular, $v_{i_j}\in\{e_1,e_2,\ldots, e_n\}$
if $j\not\in C_1\cup C_2\cup\cdots\cup C_l$. The action of $R_n$ on 
the basis vectors in \eqref{eq:1} is diagonal, that is, it is 
given, for $\sigma\in R_n$, by:
$$\sigma(v_{i_1}\otimes v_{i_2}\otimes \cdots \otimes v_{i_k},  \mathcal{D}) 
= (\sigma v_{i_1}\otimes \sigma v_{i_2}\otimes \cdots \otimes \sigma v_{i_k},  \mathcal{D}). $$

Next we explain how a vector of the form~\eqref{eq:1} 
naturally corresponds to a restricted set-partition 
and vice versa. For this, we start with recalling 
the definition of a  restricted set-partition.

For $n\geq 1$, a set-partition $S$ of a set $A$ 
containing the set $[n]$ is called an {\em $n$-restricted 
set-partition} provided that the elements of $[n]$ are 
in distinct parts of $S$. In particular, $S$ has at 
least $n$ parts. The number of $n$-restricted set-partitions 
of $[n+k]$ is called the {\em generalized Bell number},
see \cite{3}. Note that, for $n=0$, we get back the 
original Bell number $B(k)$. In the following proposition
we show that the dimension of $W_{k,n}$ is given by a 
generalized Bell number. 

\begin{proposition}\label{prop:2}
For $n\geq 1$ and $k\geq 0$, there is a one-to-one  
correspondence between $n$-restricted set-partitions 
of $[n+k]$ 
and basis vectors of the form as in \eqref{eq:1}. 
\end{proposition}

\begin{proof}
For $0\leq t\leq k$, suppose that 
$\boldsymbol{\xi}=\{B_1, B_2,\ldots, B_n, C_{n+1}, \ldots, C_{t+n}\}$ is an 
$n$-restricted set-partition of $[k+n]$ into 
$(n+t)$-parts
such that $i\in B_{i}$, for $1\leq i\leq n$. 
Consider the unique bijection $\Phi$ between the sets 
$\{n+1,n+2,\ldots, n+k\}$ and $[k]$ which preserves the
natural order. 

Let $X = \displaystyle{\bigcup_{i=1}^{n}}(B_{i}\cap \{n+1,n+2,\ldots, n+k\})$. Consider the element $v_{i_1}\otimes v_{i_2}\otimes \cdots \otimes v_{i_k}$ with the following defining property:
\begin{itemize}
\item $v_{i_s}\in \{e_1,e_2,\ldots, e_n\}$ if and only if 
$s\in \Phi(X)$. Moreover, if $s\in \Phi(X)$, then 
there exists a unique $1\leq l\leq n$ such that $s+n\in B_{l}$,
and we have $v_{i_s}=e_{l}$;
\item if $s\not\in \Phi(X)$, then $v_{i_s}=e_0$. 
\end{itemize}
We also have that $\{C_{n+1},\ldots, C_{n+t}\}$ forms a set-partition of 
the set $\{n+1,n+2,\ldots, n+k\}\setminus X$. Therefore,
$\{\Phi(C_{n+1}),\ldots, \Phi(C_{n+t})\}$ forms a set-partition of 
$[k]\setminus\Phi(X)$. Now, we can associate to our
$n$-restricted set-partition $\boldsymbol{\xi}$ the vector 
$$(v_{i_1}\otimes \cdots \otimes v_{i_k},  
\{\Phi(C_{n+1}),\ldots, \Phi(C_{n+t})\})$$ 
as in \eqref{eq:1}.

Conversely, let 
$(w_{i_1}\otimes w_{i_2}\otimes\cdots \otimes w_{i_k},\{D_1, D_2,\ldots, D_m\})$ 
be a vector of the form as in \eqref{eq:1}. For $1\leq l\leq n$, consider
$$B_{l}:=\{j+n\mid w_{i_j}=e_{l} \}\cup \{l\}. $$

Then $\{B_1,\ldots, B_n, D'_{1}, D'_{2}, \ldots,D'_{m}\}$ is an 
$n$-restricted set-partition of $[n+k]$, where $D_{i}'$ is 
obtained from $D_i$ by shifting every element by $n$. 
\end{proof}

\subsection{Partition algebras as centralizers over rook monoid}\label{sec:cent}

In the previous subsection, we constructed a basis of $W_{k,n}$ consisting 
of vectors of the form as in \eqref{eq:1}. Next we will explain that such 
a vector naturally corresponds to a partition diagram in $P_k(\delta)$.

Consider a vector 
$(v_{i_1}\otimes v_{i_2}\otimes \cdots\otimes v_{i_k}, \mathcal{D})$, 
where $\mathcal{D}= \{D_1, D_2,\ldots, D_l\}$ is a set-partition 
of a subset of $[k]$.  For $j\in[n]$, let $C_{j}=\{m\in[k]\mid i_m=j\}$. 
Note that some $C_{j}$ might be empty.
Then
$$\{C_1,C_2,\dots,C_n,D_1, D_2,\ldots, D_l\} $$
forms a set-partition of $[k]$, after removing the empty sets. We denote by
$\Bar{d}$ the corresponding symmetric set-partition diagram
which is given by
$$\Bar{d}=\{C_1\cup C_1',\ldots, C_n\cup C_n', D_1, D_1', D_2, D_2', \ldots, D_l, D_l'\}.$$ 
Note that the empty sets here 
should be removed (or ignored), if any. We label each $C_j\neq \varnothing$ 
by the basis vector $e_j$ of $\mathbb{C}^n$. Each part of $\mathcal{D}$ is 
labeled by the basis vector $e_0$ of $\mathbb{C}_{\mathrm{triv}}$.

\begin{example}\label{eq:vec}
For $n=4$, $k=7$ and the basis vector  
$(e_2\otimes e_2\otimes e_2\otimes e_0\otimes e_0\otimes 
e_3\otimes e_0, \{\{4,5\}, \{7\}\})$, 
we have the corresponding set-partition
$$\Bar{d}:=\{\{1,2,3\}\cup\{1',2',3'\},\{6\}\cup\{6'\}, 
\{4,5\}, \{4',5'\}, \{7\},\{7'\}\}. $$
\end{example}

Now we are ready to give our main definition.
For any $\delta\in \C$, a partition diagram 
$d\in \mathcal{P}_k(\delta)$ and a vector $(v_{i_1}\otimes v_{i_2} \otimes \cdots\otimes v_{i_k},\mathcal{D})$ of the form as in \eqref{eq:1},
we set:
\begin{equation}\label{eq:act}
(v_{i_1}\otimes v_{i_2} \otimes \cdots\otimes v_{i_k},\mathcal{D})\cdot d= 
\begin{cases}
    \delta^m(v_{\Bar{d}\circ d},\mathcal{A}), & \text{  if the top set-partition  of $d$   }\\
    & \text{ is finer than} \{B_1,\dots,B_n,D_1,\dots, D_l\}\\
    & \text{ and } \mathrm{rank}(\Bar{d})=\mathrm{rank}(\Bar{d}\circ d);\\
    \\
0, & \text{ otherwise}.
\end{cases}
\end{equation}
Here we have:
\begin{itemize}
\item The number $m\in\mathbb{Z}_{\geq 0}$ denotes the number of components removed from the middle in the vertical concatenation $\Bar{d}\circ d$.
\item Let $\{P_1,\ldots, P_s, Q_1, Q_2,\ldots, Q_t\}$ be the bottom set-partition of $\Bar{d}\circ d$ in which each $P_i$ is connected, by $d$, to 
some unique $C_{a_i}$ (due to our refinement condition), where $1\leq i\leq s$, 
and, moreover, each $Q_1,\ldots, Q_t$ is not connected, by $d$, to any
of the $C_j$'s (this $Q_p$ can be either a bottom part of $d$ or 
connected to exactly one of the $D_q$'s). Then we define
$$ \mathcal{A}:=  \{Q_1, Q_2,\ldots, Q_t\},\quad\text{ and }\quad 
\quad\quad v_{\Bar{d}\circ d} := e_{b_1}\otimes e_{b_2}\otimes\cdots\otimes e_{b_k}, $$
where, for $1\leq j\leq k$,
we have 
$b_j= 0$ provided that $j\in Q_{1}\cup Q_2\cup\cdots\cup Q_t$ and 
$b_j= a_i$ provided that $ j\in P_{i}$.
\end{itemize}

\begin{example}
For the element in Example~\ref{eq:vec} and the partition diagram 
$d$ as in Figure~\ref{figfig7}, we have
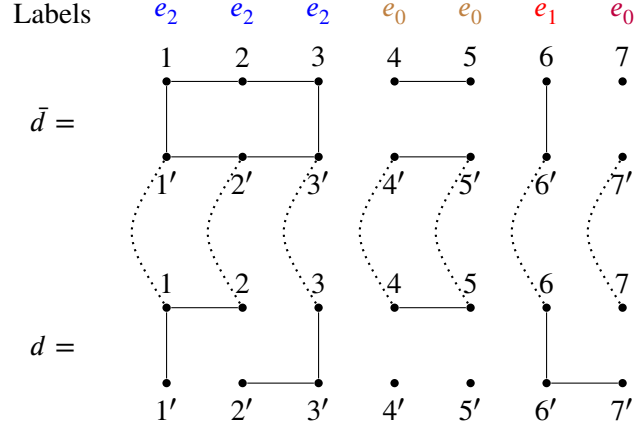
\begin{figure}
  \begin{tikzpicture}[scale=1,mycirc/.style={circle,fill=black, minimum size=0.1mm, inner sep = 1.1pt}]
\node at (-1.5,0.5) {$\Bar{d} = $};
\node at (-1.5,1.9) {Labels};
\node at (0,1.9) {\color{blue}{$e_2$}};
\node at (1,1.9) {\color{blue}{$e_2$}};
\node at (2,1.9) {\color{blue}{$e_2$}};
\node at (3,1.9) {\color{brown}{$e_0$}};
\node at (4,1.9) {\color{brown}{$e_0$}};
\node at (5,1.9) {\color{red}{$e_1$}};
\node at (6,1.9) {\color{purple}{$e_0$}};
\node[mycirc,label=above:{$1$}] (n1) at (0,1) {};
\node[mycirc,label=above:{$2$}] (n2) at (1,1) {};
\node[mycirc,label=above:{$3$}] (n3) at (2,1) {};
\node[mycirc,label=above:{$4$}] (n4) at (3,1) {};
\node[mycirc,label=above:{$5$}] (n5) at (4,1) {};
\node[mycirc,label=above:{$6$}] (n6) at (5,1) {};
\node[mycirc,label=above:{$7$}] (n7) at (6,1) {};

\node[mycirc,label=below:{$1'$}] (n1') at (0,0) {};
\node[mycirc,label=below:{$2'$}] (n2') at (1,0) {};
\node[mycirc,label=below:{$3'$}] (n3') at (2,0) {};
\node[mycirc,label=below:{$4'$}] (n4') at (3,0) {};
\node[mycirc,label=below:{$5'$}] (n5') at (4,0) {};
\node[mycirc,label=below:{$6'$}] (n6') at (5,0) {};
\node[mycirc,label=below:{$7'$}] (n7') at (6,0) {};

\path[-, draw](n1) to (n2);
\path[-, draw](n2) to (n3);
\path[-,draw](n1') to (n2'); 
\path[-,draw](n2') to (n3'); 
\path[-,draw] (n1) to (n1');
\path[-,draw] (n3) to (n3');

\path[-, draw](n4) to (n5);
\path[-,draw](n4') to (n5');  
\path[-,draw](n6) to (n6');

\node at (-1.5,-2.5) {$d = $};
\node[mycirc,label=above:{$1$}] (m1) at (0,-2) {};
\node[mycirc,label=above:{$2$}] (m2) at (1,-2) {};
\node[mycirc,label=above:{$3$}] (m3) at (2,-2) {};
\node[mycirc,label=above:{$4$}] (m4) at (3,-2) {};
\node[mycirc,label=above:{$5$}] (m5) at (4,-2) {};
\node[mycirc,label=above:{$6$}] (m6) at (5,-2) {};
\node[mycirc,label=above:{$7$}] (m7) at (6,-2) {};
\path[dotted,draw, thick] (n1')..controls(-0.6,-1)..(m1);
\path[dotted,draw, thick] (n2')..controls(0.4,-1)..(m2);
\path[dotted,draw, thick] (n3')..controls(1.4,-1)..(m3);
\path[dotted,draw, thick] (n4')..controls(2.4,-1)..(m4);
\path[dotted,draw, thick] (n5')..controls(3.4,-1)..(m5);
\path[dotted,draw, thick] (n6')..controls(4.4,-1)..(m6);
\path[dotted,draw, thick] (n7')..controls(5.4,-1)..(m7);

\node[mycirc,label=below:{$1'$}] (m1') at (0,-3) {};
\node[mycirc,label=below:{$2'$}] (m2') at (1,-3) {};
\node[mycirc,label=below:{$3'$}] (m3') at (2,-3) {};
\node[mycirc,label=below:{$4'$}] (m4') at (3,-3) {};
\node[mycirc,label=below:{$5'$}] (m5') at (4,-3) {};
\node[mycirc,label=below:{$6'$}] (m6') at (5,-3) {};
\node[mycirc,label=below:{$7'$}] (m7') at (6,-3) {};

\path[-,draw] (m6') to (m7');
\path[-, draw](m1) to (m2);
\path[-,draw](m1) to (m1'); 
\path[-,draw](m2') to (m3'); 
\path[-,draw] (m3) to (m3');
\path[-,draw] (m4) to (m5);
\path[-,draw](m6) to (m6');                
	\end{tikzpicture}
	\label{figfig7}
	\caption{Action of $\mathcal{P}_k(\delta)$ on basis vectors}
\end{figure} 
$m=1$, $\mathcal{A}=\{\{4\}, \{5\}\}$, and 
$v_{\Bar{d}\circ d}= e_2\otimes e_2\otimes 
e_2\otimes e_0\otimes e_0\otimes e_1\otimes e_1.$ 
\end{example} 

\begin{proposition}\label{prop-34}
The formula \eqref{eq:act} defines a right action of 
$\mathcal{P}_k(\delta)$ on $W_{k,n}$. This action 
commutes with the action of $R_n$.
\end{proposition}

\begin{proof}
Due to the diagrammatic nature of our definition, 
to prove that \eqref{eq:act} indeed defines a right action of 
$\mathcal{P}_k(\delta)$ on $W_{k,n}$, it is enough to 
check that 
$$ (v_{i_1}\otimes v_{i_2}\otimes \cdots \otimes v_{i_k},\mathcal{D})\cdot (d_1 d_2) = 0 \Longleftrightarrow ((v_{i_1}\otimes v_{i_2}\otimes \cdots \otimes v_{i_k},\mathcal{D})\cdot d_1)\cdot d_2 = 0,$$
for any pair $d_1$ and $d_2$ of partition diagrams.
If we assume that  $\mathrm{rank}(\Bar{d})>\mathrm{rank}(\Bar{d}\circ d_1\circ d_2)$,
then both left hand sides output $0$ derectly by \eqref{eq:act} 
and everything is fine. Therefore it remains
to consider the case when $\mathrm{rank}(\Bar{d})=\mathrm{rank}(\Bar{d}\circ d_1)
=\mathrm{rank}(\Bar{d}\circ d_1\circ d_2)$.
This splits into the following two subcases.

\textbf{Subcase 1:} Assume that the top set-partition of $d_1$ is not finer than 
$\{B_1,\dots,B_n,D_1,\dots, D_l\}$. Then the top-set-partition 
of $d_1\circ d_2$ will also not be finer than 
$\{B_1,\dots,B_n,D_1,\dots, D_l\}$. So, in this case we have $0$ on both sides,
so everything is fine.

\textbf{Subcase 2:} Suppose now that the top set-partition of 
$d_1\circ d_2$ is not finer than 
$\{B_1,\dots,B_n,D_1,\dots, D_l\}$ but the top-set 
partition of $d_1$ is finer than  
$\{B_1,\dots,B_n,D_1,\dots, D_l\}$. Then there exists a
part $C$ of $\mathrm{top}(d_1\circ d_2)$ which intersect 
at least two parts of $\{B_1,\dots,B_n,D_1,\dots, D_l\}$. 
Further, this part $C$ 
must be contained in a part of $d_1\circ d_2$ obtained by combining at least two distinct bottom parts of $d_1$, all intersecting a top part of $d_2$ in the concatenation $d_1\circ d_2$. 
This is equivalent to saying that $\mathrm{top}(d_2)$ is not 
finer than $\mathrm{bottom}(\Bar{d}\circ d_1)$, which is also 
the set-partition obtained from the vector 
$$(v_{i_1}\otimes v_{i_2}\otimes \cdots \otimes v_{i_k},\mathcal{D})\cdot d_1.$$ 
This completes the proof in this subcase.

It follows directly from the definitions that the action defined by
\eqref{eq:act} commutes with the action of $R_n$ on $W_{k,n}$.
This completes the proof.
\end{proof}

From Proposition~\ref{prop-34}, it follows that $W_{k,n}$ is an 
$R_n$-$P_k(\delta)$-bimodule. Our next essential step in the proof
of the main Theorem~\ref{thm:key} is computation of the 
bitrace of the bimodule $W_{k,n}$.

\subsection{Set partitions stabilized by permutations}\label{newsec-composition}

For a positive integer $m$, consider the natural action of 
the symmetric group $S_m$ on the set $\mathcal{SP}_m$
of all set partitions of $[m]$.

Let $\mu=(\mu_1,\mu_2,\ldots,\mu_l)$ be
a composition of $m$. In particular, $0<\mu_i$, for all $i$,
and $\mu_1+\mu_2+\dots+\mu_l=m$. For $i=1,2,\dots,l$,
we denote by $a_i$ the number $\mu_1+\mu_2+\dots+\mu_{i-1}$,
where $a_1=0$.
Let $\sigma_\mu$ be the 
element of $S_m$ defined as the following product of cycles:
\begin{displaymath}
\sigma_\mu:=(1,2,\dots,\mu_1)(\mu_1+1,\mu_1+2,\dots,\mu_1+\mu_2)
\dots (m-\mu_l+1,m-\mu_l+2,\dots,m).
\end{displaymath}
For this $\sigma_\mu$, we have the set $\mathcal{SP}_m^\mu$
of all elements in $\mathcal{SP}_m$ that are fixed by
$\sigma_\mu$. 

The set $\mathcal{SP}_m^\mu$ can be described fairly explicitly,
but not in a very simple way. Let $\xi$ be a set partition
of $[l]$. Let $f$ be a function from the parts of $\xi$
to positive integers that satisfies the following property:
given a part $P$ of $\xi$, the value $f(P)$ divides
$\mu_i$, for all $i\in P$. We denote the set of all such
functions by $Q(\mu,\xi)$. 

For each part $P$ of 
$\xi$ and each positive divisor $d$
of all $\mu_i$, where $i\in P$, denote by
$R(P,d)$ the set of all functions $g$ that assigns to
an elements $i\in P$  an integer value $1\leq g(i)\leq d$.
We additionally assume that $g(i_0)=1$ for
the minimum element $i_0\in P$.

Assume now that we have fixed the following datum:
$(\xi,f,\{g_P\})$, where $\xi$ is a set partition of $[l]$,
$f\in Q(\mu,\xi)$ and $g_P\in R(P,f(P))$, for every part
$P$ of $\xi$. To this datum we associate the
set partition $\eta(\xi,f,\{g_P\})$ of $[m]$ 
constructed as follows. For a part $P$ of $\xi$, set
\begin{displaymath}
N_P:=\bigcup_{i\in P} 
\{a_i+g_P(i),a_i+g_P(i)+d,a_i+g_P(i)+2d,\dots,a_i+g_P(i)+\mu_i-d\}.
\end{displaymath}
The partition $\eta(\xi,f,\{g_P\})$ consists exactly of all parts of
the form $\sigma_\mu^r(N_P)$, where $P$ is a part of $\xi$ and  $r\geq 1$.
Note that, by construction, we have $\eta(\xi,f,\{g_P\})\in \mathcal{SP}_m^\mu$.

\begin{lemma}\label{lem-n-part}
Each element in $\mathcal{SP}_m^\mu$ can be uniquely written
as $\eta(\xi,f,\{g_P\})$, for some $\xi,f$ and $\{g_p\}$ as above.
\end{lemma}

Define 
$Q_1:=\{1,2,\dots,\mu_1\}$, then
$Q_2:=\{\mu_1+1,\mu_1+2,\dots,\mu_1+\mu_2\}$ and so on
all the way up to $Q_l:=\{\mu_1+\dots+\mu_{l-1}+1,\dots,m\}$.

\begin{proof}
Let $\eta\in \mathcal{SP}_m^\mu$ and let $K$ be the part of 
$\eta$ that contains $1$. Define the part $P$ of $\xi$ containing
$1$ as the collection of all $j$ for which $K\cap Q_j\neq\varnothing$.
Define $f(P)$ as the number of elements in the set
$\{K,\sigma_\mu(K),\sigma_\mu^2(K),\dots\}$.
For each $j\in P$, define $g_P(j)$ as the minimal element in $P\cap Q_j$.

If $P\neq [l]$, take the minimal element $i_0\in [l]\setminus P$.
Let $K_1$ be the part of $\eta$ that contains 
$\mu_1+\dots+\mu_{i_0-1}+1$. Define the part $P_1$ of $\xi$ containing
$i_0$ as the collection of all $j$ for which $K_1\cap Q_j\neq\varnothing$,
and so on. Proceeding inductively, defines $\eta$, $f$ and all $g_P$
that give $\eta=\eta(\xi,f,\{g_P\})$. Uniqueness follows directly from the
construction.
\end{proof}

As an example, consider $m=6$ and $\mu=(2,4)$. 
Then $a_1=0$ and $a_2=1$. Let $\xi=\{1,2\}$
and $f(\{1,2\})=2$. Note that $2\vert 2$ and $2\vert 4$.
Take as $g_{\{1,2\}}$ the function that sends both $1$ and $2$ to $1$.
Then the corresponding $\eta(\xi,f,\{g_P\})$ equals
$\{\{1,3,5\},\{2,4,6\}\}$. Here $N_{\{1,2\}}=\{1,3,5\}$.
If we take as $g_{\{1,2\}}$ the function that sends $2$ to $2$,
Then the corresponding $\eta(\xi,f,\{g_P\})$ equals
$\{\{1,4,6\},\{2,3,5\}\}$. Here $N_{\{1,2\}}=\{1,4,6\}$.

\subsection{Dual symmetric inverse monoid algebras
as subalgebras of partition algebras}\label{sec:swd}

Let $X$ and $Y$ be two disjoint subsets of $[k]$.
Set $Z:=[k]\setminus(X\cup Y)$.
Let $\xi$ be a set partition of $X$.
Let $Y=\{y_1,y_2,\dots,y_r\}$ and let
$\xi$ have parts $A_1,A_2,\dots,A_s$.
Denote by $\mathcal{I}(X,Y,\xi)$ the subspace of 
$\mathcal{P}_k(\delta)$ spanned by all partition 
diagrams  of the form
\begin{displaymath}
\{A_1\cup A'_1,\dots, A_s\cup A'_s,
\{y_1\},\dots,\{y_r\},\{y'_1\},\dots,\{y'_r\}, B\}, 
\end{displaymath}
where $B$ is an arbitrary set partition of 
$Z\cup Z'$ 
in which each part is propagating.

\begin{lemma}\label{lem-dualinpartition}
The subspace $\mathcal{I}(X,Y,\xi)$ is a subalgebra of 
$\mathcal{P}_k(\delta)$. Moreover,
if $\delta\neq 0$, then sending a partition diagram $\xi$
of the above form to $\frac{1}{\delta^{|Y|}}$ times the restriction
of $\xi$ to $Z\cup Z'$ defines an isomorphism between
$\mathcal{I}(X,Y,\xi)$ and the monoid algebra of $I_Z^*$.
\end{lemma}

\begin{proof}
It is clear from the definitions that vertical concatenation
of partition diagrams in $\mathcal{I}(X,Y,\xi)$ gives a 
partition diagram in $\mathcal{I}(X,Y,\xi)$. This implies
the first claim.

It is also clear from the definitions that concatenation
of partition diagrams in $\mathcal{I}(X,Y,\xi)$
results in the factor $\delta^{|Y|}$ for multiplication
inside $\mathcal{P}_k(\delta)$. Therefore the rescaling 
by $\frac{1}{\delta^{|Y|}}$ will compensate for this 
additional factor. This implies that, after this rescaling,
the product of diagrams inside $\mathcal{P}_k(\delta)$
agrees with the multiplication in the monoid algebra of 
$I_Z^*$. This implies the second statement.
\end{proof}

Set $t=k-|X|-|Y|$. Let $V_{t,\xi}$ be a subspace of $W_{k,n}$ 
spanned by all vectors of the form 
$$(v_{i_1}\otimes \cdots\otimes v_{i_k}, \{\xi,\{y_1\},\dots,\{y_r\}\}),$$
where $\xi$ is fixed.
The space $V_{t,\xi}$ inherits an action of the 
subalgebra $\mathcal{I}(X,Y,\xi)$ from the 
action of the partition algebra $\mathcal{P}_k(\delta)$ 
on $W_{k,n}$. Clearly, $V_t$ is also an  $R_n$-module. 

Note that, when $X\cup Y$ is empty, 
the algebra $\mathcal{I}(X,Y,\xi)$ is 
the monoid algebra $\mathbb{C}[\mathcal{I}_k^*]$ of the dual symmetric 
inverse monoid $\mathcal{I}_k^*$. 

\begin{proposition}\label{prop:mild} 
If $\delta\neq 0$, then the above map
$\phi_{(X,Y,\xi)}: \mathcal{I}(X,Y,\xi) \to \End_{R_n}(V_{t,\xi})$ is surjective.
\end{proposition}

\begin{proof}
This follows by by combining  
Lemma~\ref{lem-dualinpartition} with \cite[Theorem~1]{5}.
\end{proof}

\subsection{Representations of $\mathcal{I}(X,Y,\xi)$}

From Proposition~\ref{prop:mild},
we see that, when $\delta\neq 0$, the irreducible 
representations of $\mathcal{I}(X,Y,\xi)$ are indexed by 
partitions of positive integers that do not exceed $t$. 
Also, these representations can be constructed as
described in Section~\ref{dualreps}. 

Denote by $\mathcal{K}_{(X,Y,\xi)}(t,i)$ and $\mathcal{L}_{(X,Y,\xi)}(t,i)$
the sets of partition diagrams for $\mathcal{I}(X,Y,\xi)$
which correspond to the sets 
$\mathcal{K}(t,i)$ and $\mathcal{L}(t,i)$ from Section~\ref{dualreps}
via the isomorphism given by Proposition~\ref{prop:mild}.
Then 
$$\{d\otimes v_T\mid d\in\mathcal{K}_{(X,Y,\xi)}(t,i), T\in\text{SYT}(\lambda)\}$$
is a basis of the simple $\mathcal{I}(X,Y,\xi)$-module $\mathcal{I}_{(X,Y,\xi)}^\lambda$
corresponding to $\lambda\vdash i\leq t$.

Another immediate consequence of Proposition~\ref{prop:mild} is the 
following multiplicity-free decomposition,  
as  an $R_n$-$\mathcal{I}(X,Y,\xi)$-bimodule:  
\begin{equation}
V_{t,\xi} \cong \bigoplus_{1\leq |\lambda|\leq \min\{n,t\}} 
R_{n}^{\lambda}\otimes \mathcal{I}_{(X,Y,\xi)}^{\lambda}.
\end{equation}

\subsection{Characters of partition algebras in terms of characters of dual symmetric inverse monoid algebras}\label{sub:Cr}

Let $\chi_{\mathcal{I}_{(X,Y,\xi)}^\lambda}$ denote the character of the 
irreducible representation $\mathcal{I}_{(X,Y,\xi)}^{\lambda}$ 
of $\mathcal{I}(X,Y,\xi)$. 
Recall that $\chi_{\mathcal{P}_{k}^\lambda}$ denotes the character 
of the irreducible representation $\mathcal{P}_{k}^\lambda$ of the 
partition algebra $\mathcal{P}_k(\delta)$. 

For $m\leq k$, consider the case when $Y=[k]\setminus [m]$
and $X\subset [m]$. Let $\xi$ be any set-partition of $X$.
For a composition $\mu$ of $m$, we have the element
$d_\mu$ defined in Subsection~\ref{sec:char}. 
Assuming that the restriction of $d_\mu$ to $X$ fixes $\xi$,
we define
the element $d_{\mu}^{X,\xi}$ as follows:
\begin{itemize}
\item $d_{\mu}^{X,\xi}$  coincides with $d_\mu$ outside
$X\cup X'$;
\item on $X\cup X'$, the parts of $d_{\mu}^{X,\xi}$
are exactly the sets of the form $K\cup K'$, where
$K$ is a part of $\xi$.
\end{itemize}
It follows directly from the definition that 
$d_{\mu}^{X,\xi}\in\mathcal{I}(X,Y,\xi)$.

\begin{theorem}\label{thm:cr}
For $\delta\in\mathbb{C}\setminus\{0\}$,
a composition $\mu$ of $m\leq k$, and
$\lambda\vdash j\leq m$, we have
$$ 
\chi_{\mathcal{P}_{k}^\lambda}(d_\mu)=
\sum_{X\subset [m]}
\sum_{\xi\in\mathcal{SP}_X^\mu}
\chi_{\mathcal{I}_{(X,Y,\xi)}^\lambda}(d_\mu^{X,\xi}),$$
where, by convention, we set 
$\chi_{\mathcal{I}_{(X,Y,\xi)}^\lambda}(d_\mu^{X,\xi})=0$ provided that 
$j>t:=m-|X|$. 
\end{theorem}

\begin{proof}
Recall that $\mathcal{B}^{\lambda} = 
\{d\otimes v_{T}\mid \, d\in\mathcal{N}(k, j)\text{ and } 
T\in \text{ SYT$(\lambda)$}\}$ is a basis of $\mathcal{P}_k^\lambda$. 
By definition, $\chi_{\mathcal{P}_k^\lambda}(d_\mu)$ can be computed
as the sum, taken over all $d\otimes v_{T}\in \mathcal{B}^{\lambda}$,
with the corresponding summand being
the coefficient at $d\otimes v_T$ when expressing
$d_\mu (d\otimes v_T)$ in the basis $\mathcal{B}^{\lambda}$.
In order for $d\otimes v_T$ to appear with nonzero coefficient 
in $d_\mu(d\otimes v_T)$, it is necessary that 
$d_\mu \circ d = d \circ \sigma$,
for some $\sigma\in S_{|\lambda|}$. This necessary condition gives, 
from the concatenation rule for partition diagrams, the 
following conditions on th blocks of $d$:
\begin{itemize}
\item $\{m+1\}, \ldots, \{k\}$ must be parts of the diagram $d$;
\item the collection of the top non-propagating parts of $d$
inside $[m]$ must be $d_\mu$ invariant (and hence we can apply to 
it the description from Subsection~\ref{newsec-composition});
\item the collection of the top shadows of the
propagating parts of $d$ must be $d_\mu$ invariant (and hence we 
can apply to it the description from Subsection~\ref{newsec-composition}).
\end{itemize}
In the remainder of the proof, we consider only those $d\in\mathcal{N}(k,j)$ 
which satisfy all the above conditions. Let $X$ be the union of 
the top non-propagating parts of $d$ inside $[m]$ and
$\xi$ be the set partition of $X$ given by these 
top non-propagating parts of $d$. Then, directly from the definitions,
we have that
both $d_\mu \circ d=d_\mu^{X,\xi} \circ d$
and $d_\mu d=d_\mu^{X,\xi} d$ (the difference between the two 
identities is the factor $\delta^{k-m}$ appearing on both sides
of the second identity).

Define $d^{X,\xi}$ as follows:
\begin{itemize}
\item $d^{X,\xi}$ has parts $\{s\}$ and $\{s'\}$, for each $s>m$;
\item $d^{X,\xi}$ has parts $K\cup K'$, for any part $K$ of $\xi$;
\item if $K_1,K_2,\dots,K_r$ is the list of top shadows of the 
propagating parts of $d$ with respect to the minimal element order,
then the remaining propagating parts of $d^{X,\xi}$ are:
\begin{displaymath}
K_1\cup\{a_1'\},K_2\cup\{a_2'\},\dots, K_{r-1}\cup\{a_{r-1)}'\}
\text{ and } K_r\cup K',
\end{displaymath}
where $a_1,a_2,\dots,a_{r-1}$ are the $r-1$ smallest elements in
$[m]\setminus X$ written in the increasing order 
$K=[m]\setminus (X\cup \{a_1,a_2,\dots,a_{r-1}\})$.
\end{itemize}
An example of both $d_\mu^{X,\xi}$ and $d^{X,\xi}$ can be found in
Example~\ref{ex:ex}. 
Directly from the definition, we see that 
$d^{X,\xi}\in \mathcal{K}_{(X,Y,\xi)}(t, j)$.

From the construction, we have that 
$d_\mu \circ d = d \circ \sigma$ implies
$d_\mu^{X,\xi}\circ d^{X,\xi}=d^{X,\xi}\circ \sigma$,
for the same $\sigma$.
Therefore, 
the contribution of $d_\mu d\otimes v_{T}$ to the computation of
$\chi_{\mathcal{P}_k^\lambda}(d_\mu)$ coincides with
the contribution of $d_\mu^{X,\xi} d^{X,\xi}\otimes v_{T}$
to the computation of 
$\chi_{\mathcal{I}_{(X,Y,\xi)}^\lambda}(d_\mu^{X,\xi})$.
As any possible cross-contributions result in $0$ since the number
of propagating lines in $d_\mu \circ d$ strictly decreases compared to
the number of propagating lines in $d$,
the claim of the theorem follows.
\end{proof}

\begin{example}\label{ex:ex}
Let $k=12$ and $\mu = (2,3,1,2,2)\models 10$. Consider $d\in\mathcal{N}(12, 2)$ given by:

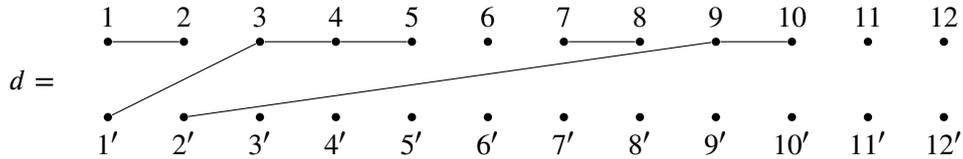
\begin{center}
\begin{tikzpicture}[scale=1,mycirc/.style={circle,fill=black, minimum size=0.1mm, inner sep = 1.1pt}]
\node at (-1,0.5) {$d=$};
\node[mycirc,label=above:{$1$}] (n1) at (0,1) {};
   \node[mycirc,label=above:{$2$}] (n2) at (1,1) {};
   \node[mycirc,label=above:{$3$}] (n3) at (2,1) {};
   \node[mycirc,label=above:{$4$}] (n4) at (3,1) {};
   \node[mycirc,label=above:{$5$}] (n5) at (4,1) {};
   \node[mycirc,label=above:{$6$}] (n6) at (5,1) {};
   \node[mycirc,label=above:{$7$}] (n7) at (6,1) {};
   \node[mycirc,label=above:{$8$}] (n8) at (7,1) {};
   \node[mycirc,label=above:{$9$}] (n9) at (8,1) {};
   \node[mycirc,label=above:{$10$}] (n10) at (9,1) {};
   \node[mycirc,label=above:{$11$}] (n11) at (10,1) {};
   \node[mycirc,label=above:{$12$}] (n12) at (11,1) {};

   \node[mycirc,label=below:{$1'$}] (n1') at (0,0) {};
				\node[mycirc,label=below:{$2'$}] (n2') at (1,0) {};
				\node[mycirc,label=below:{$3'$}] (n3') at (2,0) {};
				\node[mycirc,label=below:{$4'$}] (n4') at (3,0) {};
				\node[mycirc,label=below:{$5'$}] (n5') at (4,0) {};
				\node[mycirc,label=below:{$6'$}] (n6') at (5,0) {};
				\node[mycirc,label=below:{$7'$}] (n7') at (6,0) {};
				\node[mycirc,label=below:{$8'$}] (n8') at (7,0) {};
                \node[mycirc,label=below:{$9'$}] (n9') at (8,0) {};
				\node[mycirc,label=below:{$10'$}] (n10') at (9,0) {};
				\node[mycirc,label=below:{$11'$}] (n11') at (10,0) {};
				\node[mycirc,label=below:{$12'$}] (n12') at (11,0) {};
          \path[-, draw](n1) to (n2);
				\path[-,draw](n3) to (n4) to (n5);  
                \path[-, draw](n3) to (n1');
				\path[-,draw](n7) to (n8);
                \path[-, draw](n9) to (n10);
				\path[-, draw](n9) to (n2');
\end{tikzpicture}
\end{center}   
Then $Y=\{11,12\}$, $X=\{1,2,6,7,8\}$ and $\xi=\{\{1,2\}, \{6\}, \{7,8\}\}$, 
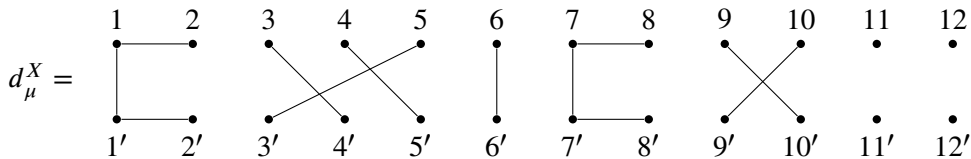
\begin{center}
\begin{tikzpicture}[scale=1,mycirc/.style={circle,fill=black, minimum size=0.1mm, inner sep = 1.1pt}]
\node at (-1,0.5) {$d_{\mu}^X=$};
\node[mycirc,label=above:{$1$}] (n1) at (0,1) {};
   \node[mycirc,label=above:{$2$}] (n2) at (1,1) {};
   \node[mycirc,label=above:{$3$}] (n3) at (2,1) {};
   \node[mycirc,label=above:{$4$}] (n4) at (3,1) {};
   \node[mycirc,label=above:{$5$}] (n5) at (4,1) {};
   \node[mycirc,label=above:{$6$}] (n6) at (5,1) {};
   \node[mycirc,label=above:{$7$}] (n7) at (6,1) {};
   \node[mycirc,label=above:{$8$}] (n8) at (7,1) {};
   \node[mycirc,label=above:{$9$}] (n9) at (8,1) {};
   \node[mycirc,label=above:{$10$}] (n10) at (9,1) {};
   \node[mycirc,label=above:{$11$}] (n11) at (10,1) {};
   \node[mycirc,label=above:{$12$}] (n12) at (11,1) {};

   \node[mycirc,label=below:{$1'$}] (n1') at (0,0) {};
				\node[mycirc,label=below:{$2'$}] (n2') at (1,0) {};
				\node[mycirc,label=below:{$3'$}] (n3') at (2,0) {};
				\node[mycirc,label=below:{$4'$}] (n4') at (3,0) {};
				\node[mycirc,label=below:{$5'$}] (n5') at (4,0) {};
				\node[mycirc,label=below:{$6'$}] (n6') at (5,0) {};
				\node[mycirc,label=below:{$7'$}] (n7') at (6,0) {};
				\node[mycirc,label=below:{$8'$}] (n8') at (7,0) {};
                \node[mycirc,label=below:{$9'$}] (n9') at (8,0) {};
				\node[mycirc,label=below:{$10'$}] (n10') at (9,0) {};
				\node[mycirc,label=below:{$11'$}] (n11') at (10,0) {};
				\node[mycirc,label=below:{$12'$}] (n12') at (11,0) {};
\path[-, draw](n1) to (n2);
\path[-, draw](n1') to (n2');
\path[-,draw] (n1) to (n1');
\path[-,draw](n3) to (n4');  
\path[-,draw](n4) to (n5');
\path[-,draw] (n5) to (n3');
\path[-,draw] (n6) to (n6');
\path[-,draw] (n7) to (n8);
\path[-,draw] (n7') to (n8');
\path[-,draw] (n7) to (n7');
\path[-,draw] (n9) to (n10');
\path[-,draw] (n9') to (n10);
\end{tikzpicture}
\end{center}
and 
  \begin{center}  \begin{tikzpicture}[scale=1,mycirc/.style={circle,fill=black, minimum size=0.1mm, inner sep = 1.1pt}]
  \node at (-1,0.5) {$d^X=$};
\node[mycirc,label=above:{$1$}] (n1) at (0,1) {};
   \node[mycirc,label=above:{$2$}] (n2) at (1,1) {};
   \node[mycirc,label=above:{$3$}] (n3) at (2,1) {};
   \node[mycirc,label=above:{$4$}] (n4) at (3,1) {};
   \node[mycirc,label=above:{$5$}] (n5) at (4,1) {};
   \node[mycirc,label=above:{$6$}] (n6) at (5,1) {};
   \node[mycirc,label=above:{$7$}] (n7) at (6,1) {};
   \node[mycirc,label=above:{$8$}] (n8) at (7,1) {};
   \node[mycirc,label=above:{$9$}] (n9) at (8,1) {};
   \node[mycirc,label=above:{$10$}] (n10) at (9,1) {};
   \node[mycirc,label=above:{$11$}] (n11) at (10,1) {};
   \node[mycirc,label=above:{$12$}] (n12) at (11,1) {};

   \node[mycirc,label=below:{$1'$}] (n1') at (0,0) {};
				\node[mycirc,label=below:{$2'$}] (n2') at (1,0) {};
				\node[mycirc,label=below:{$3'$}] (n3') at (2,0) {};
				\node[mycirc,label=below:{$4'$}] (n4') at (3,0) {};
				\node[mycirc,label=below:{$5'$}] (n5') at (4,0) {};
				\node[mycirc,label=below:{$6'$}] (n6') at (5,0) {};
				\node[mycirc,label=below:{$7'$}] (n7') at (6,0) {};
				\node[mycirc,label=below:{$8'$}] (n8') at (7,0) {};
                \node[mycirc,label=below:{$9'$}] (n9') at (8,0) {};
				\node[mycirc,label=below:{$10'$}] (n10') at (9,0) {};
				\node[mycirc,label=below:{$11'$}] (n11') at (10,0) {};
				\node[mycirc,label=below:{$12'$}] (n12') at (11,0) {};
                \path[-, draw](n1) to (n2);
\path[-, draw](n1') to (n2');
\path[-,draw] (n1) to (n1');
\path[-,draw] (n6) to (n6');
\path[-,draw] (n7) to (n8);
\path[-,draw] (n7') to (n8');
\path[-,draw] (n7) to (n7');
\path[-,draw](n3) to (n4) to (n5);  
\path[-, draw](n3) to (n3');
\path[-,draw] (n9) to (n4');
\path[-,draw] (n9) to (n10);
\path[-,draw](n9') to (n10');
\path[-,draw](n4') to (n5');
\draw[-,draw] (n5') .. controls(6.5,-1.5)..(n9');
\end{tikzpicture}

                \end{center}
\end{example}

\subsection{Partition algebras as centralizers over rook monoid}

For $\sigma\in R_n$ and $d\in P_k(\delta)$, let 
$$\langle \sigma\cdot (v_{i_1}\otimes\cdots\otimes v_{i_k}, 
\mathcal{D})\cdot d, (v_{i_1}\otimes\cdots\otimes v_{i_k}, \mathcal{D})\rangle$$
denote the coefficient at $(v_{i_1}\otimes\cdots\otimes v_{i_k}, \mathcal{D})$
when expressing 
$\sigma\cdot(v_{i_1}\otimes\cdots\otimes v_{i_k}, \mathcal{D})\cdot d$ in
the basis given by \eqref{eq:1}. Then the bitrace of $(\sigma, d)$, for the 
$R_n$-$P_k(\delta)$-bimodule $W_{k,n}$, is defined as follows:
$$\bitrace(\sigma, d):= \sum_{(v_{i_1}\otimes\cdots\otimes v_{i_k}, \mathcal{D})} \langle \sigma\cdot (v_{i_1}\otimes\cdots\otimes v_{i_k}, \mathcal{D})\cdot d, (v_{i_1}\otimes\cdots\otimes v_{i_k}, \mathcal{D})\rangle. $$

\begin{theorem}\label{thm:key}
For $\delta\in \C$ such that $\mathcal{P}_k(\delta)$ is semi-simple, 
we have the decomposition
$$W_{k,n} \cong 
\bigoplus_{0\leq i\leq \min\{n,k\}}
\bigoplus_{\lambda\vdash i} R_n^{\lambda}\otimes P_k^{\lambda},   $$ 
as $R_n$-$P_k(\delta)$-bimodules.
\end{theorem}

\begin{proof}
Suppose that $\delta\in\C$ such that $\mathcal{P}_k(\delta)$ is semi-simple,
in particular, $\delta\neq 0$. Then both $\mathcal{P}_k(\delta)$
and $\mathbb{C}[R_n]$ are semi-simple and hence, to prove our theorem, 
it is enough to check that, for $\sigma\in R_n$ and $d\in \mathcal{P}_k(\delta)$,
we have
\begin{displaymath}
\bitrace(\sigma,d)= \sum_{\substack{\lambda\vdash i\\ 0\leq i\leq \min\{n,k\}}}\chi^{\lambda}_{R_n}(\sigma)\chi^{\lambda}_{\mathcal{P}_k(\delta)}(d). 
\end{displaymath}
From Section~\ref{sec:char} (see~\cite{6} more detail), recall that, for a 
given partition diagram $d\in \mathcal{P}_{k}(\delta)$, there is a
partition diagram $d_{\mu}$ indexed by some composition 
$\mu=(\mu_1,\ldots,\mu_l)$ satisfying~\eqref{eq:inde}. 
For a fixed $\sigma\in R_n$, the number $\bitrace(\sigma, d)$ 
(resp. $\bitrace(\sigma, d_{\mu})$) is the character for 
the action of $\mathcal{P}_k(\delta)$ on $\sigma\cdot W_{k,n}$ 
evaluated at $d$ (resp. $d_{\mu}$). Hence, from Section~\ref{sec:char}, 
it is enough to prove that
\begin{equation}\label{eq:des}
 \bitrace(\sigma,d_{\mu})= \sum_{\substack{\lambda\vdash i\\ 0\leq i\leq \min\{n,k\}}}\chi^{\lambda}_{R_n}(\sigma)\chi^{\lambda}_{\mathcal{P}_k(\delta)}(d_{\mu}).
 \end{equation}
 From the definitions, we can split the right hand side of
 \eqref{eq:des} as $A + B$, where 
 $$ A = \sum_{\substack{U \text{ is a set-partition } \\ \text{ of } [k]}} \langle (\sigma\cdot (e_0\otimes\cdots\otimes e_0, U)\cdot d_{\mu}, (e_0\otimes\cdots\otimes e_0, U)\rangle, \text{ and } $$
and
 $$B=\sum_{\substack{U \text{ is a set-partition of} \\ \text{ a proper subset of } [k]}} \,\,\sum_{(v_{i_1}\otimes\cdots\otimes v_{i_k}, \, U)} \langle (\sigma\cdot (v_{i_1}\otimes\cdots\otimes v_{i_k}, U)\cdot d_{\mu}, (v_{i_1}\otimes\cdots\otimes v_{i_k}, X)\rangle. $$
 Note that $\sigma\in R_n$ fixes $e_0$ and hence the sum $A$ is equal to 
 $\chi_{\mathcal{P}^\varnothing_k}(d_\mu)$ which is, clearly, equal to $\chi_{R^{\varnothing}}(\sigma)\chi_{\mathcal{P}^\varnothing_k}(d_\mu)$. 

Further, in order for $(v_{i_1}\otimes \cdots\otimes v_{i_k}, U)$ to appear 
with a non-zero coefficient in 
$\sigma\cdot(v_{i_1}\otimes \cdots\otimes v_{i_k}, U)\cdot d_{\mu}$ when expressed as a linear combination of vectors of the form~\eqref{eq:1}, the set-partition $X$ 
has to satisfy the following:
\begin{itemize}
\item for all $j>m$, the singleton $\{j\}$ must be a part of $U$;
\item the remaining parts of $U$ form a  $\mu$-invariant collection.
\end{itemize}
Denote $[k]\setminus[m]$ by $Y$, then the union of all parts of $U$ minus $Y$ 
by $X$ and the set partition of $X$ induced by $U$ by $\xi$. Then we have
the associated subalgebra $\mathcal{I}(X,Y,\xi)$ from Subsection~\ref{sec:swd}.
We also have the partition diagram $d_{\mu}^{X,\xi}\in \mathcal{I}(X,Y,\xi)$,
defined in  Subsection~\ref{sub:Cr}, such that
$$ (v_{i_1}\otimes\cdots\otimes v_{i_k}, X)\cdot d_{\mu}= (v_{i_1}\otimes\cdots\otimes v_{i_k}, X)\cdot d_{\mu}^{X,\xi}.$$
Combining this with the Schur--Weyl dualities from Subsection~\ref{sec:swd} and using the convention that, for $|\lambda|>t$, we have
$\chi_{\mathcal{I}^{\lambda}_{(X,Y,\xi)}}(d_{\mu}^{X,\xi}) = 0$, we obtain the following equality: 
$$B = \sum_{X\subset [m]}\sum_{\xi\in\mathcal{SP}_X^\mu}\,\,\,\sum_{1\leq |\lambda|\leq\min\{n,k\}}\chi_{R^{\lambda}}(\sigma)\chi_{\mathcal{I}_{(X,Y,\xi)}^{\lambda}}(d_{\mu}^{X,\xi}).$$
Using Theorem~\ref{thm:cr}, we can rewrite this as
$$B= \sum_{1\leq |\lambda|\leq \min\{n,k\}} \chi_{{R}^{\lambda}}(\sigma)
\chi_{\mathcal{P}^{\lambda}_k}(d_\mu).$$
Adding $A$ and $B$, we get our desired expression~\eqref{eq:des}.
\end{proof}   

Theorem~\ref{thm:key} and the fact $\dim \End_{R_n}(W_{k,n})\leq B(2k)$, with equality if $n\geq k$, imply the following statement, 
which is the main result of this paper.

\begin{theorem}\label{thm:main2}
 If $\mathcal{P}_k(\delta)$ is semi-simple, the action map $\varphi: \mathcal{P}_k(\delta) \to \End_{R_n}(W_{k,n})$ 
is surjective. In case $n\geq k$ and $\delta\neq 0$, the map is surjective iff $\mathcal{P}_k(\delta)$ is 
semi-simple.
\end{theorem}

\begin{example}
Let $k=1$. Then $W_{k,n}= \C^n\oplus \C_{\triv}$. 
Both $\mathcal{P}_k(\delta)$ and $\End_{R_n}(W_{k,n})$ are $2$-dimensional. 
Let $d_1=\{\{1\}, \{1'\}\}$ and $d_2=\{\{1,1'\}\}$. We will compute 
$\phi(d_1)$ and $\phi(d_2)$. Then 
$$\phi(d_1)(e_j)=
\begin{cases}
\delta e_0,   & \text{ if } j=0;\\
0, & \text{ if } 1\leq j\leq n;
\end{cases} \quad\text{ and } \quad
\phi(d_2)(e_j)= e_j \text{ for all } 0\leq j\leq n.$$
In this case the map $\phi$ is surjective iff $\delta\neq 0$. 
Note that $\delta=0$ is the only value for which $\mathcal{P}_1(\delta)$
is not semi-simple.
\end{example}

\subsection{RSK type bijection}

\begin{proposition}\label{prop-abc5}
There is a bijection between the set of all $n$-restricted set-partitions of 
$[n+k]$ and the set consisting of all triples $((P,Q), T)$,
where 
\begin{itemize}
\item $P$ is a standard set-partition tableau with disjoint entries,
the union of entries of which is a subset of $[k]$,
\item $Q$ is an standard tableau of the same shape as $P$
with entries from $[n]$,
\item $T$ is a standard row-shape set-partition tableau
with disjoint entries, the union of the entries of which 
complement the union of the entries of $P$ to $[k]$.
\end{itemize}
\end{proposition}

\begin{proof}
By Proposition~\ref{prop:2}, an $n$-restricted set-partition of $[n+k]$ 
uniquely corresponds to a vector 
$(v_{i_1}\otimes v_{i_2}\otimes\cdots\otimes v_{i_k}, \mathcal{D})$ 
of the form \eqref{eq:1}, 
where  $\mathcal{D}=\{C_1, C_2,\ldots, C_l\}$ is a set-partition of a 
subset of $[k]$. Assume that
$$\min C_1<\min C_2 < \min< \cdots <\min C_l.$$
We can now define $T$ as the row-shape tableau whose $j$-th entry 
is $C_j$, for $1\leq j\leq l$. 

Further, from the vector 
$v_{i_1}\otimes v_{i_2}\otimes\cdots\otimes v_{i_k}$, we can read off 
the following information:
\begin{equation}\label{eqnnd45}
 \begin{pmatrix}
1 & 2 & 3 &\ldots & n\\
G_1 & G_2 & G_3 & \ldots & G_n
\end{pmatrix},
\end{equation}
where $G_s$ consists of all $t$, for which $i_t=s$.
We delete all columns, for which $G_i=\varnothing$.
We also note that the remaining non-empty $G_s$'s 
are totally ordered with respect to the minimum entry 
order.

Now we can apply the usual Robinson-Schensted bijection
to this reduced version of \eqref{eqnnd45}. 
The insertion tableau that we get will be a 
standard set-partition tableau $P$ with disjoint entries.
The union of these entries is a subset of $[k]$. 
The recording tableau $Q$ will be a standard tableau of 
the same shape as $P$ and with entries from $[n]$. 
It is easy to see that this whole procedure is reversible
and hence bijective. 
\end{proof}

As an immediate consequence of Proposition~\ref{prop-abc5}, we obtain:

\begin{corollary}
The number of $n$-restricted set-partitions of $[n+k]$ is equal to 
$$ \sum_{\substack{\lambda\vdash i\\ 0\leq i\leq \min\{n,k\}}} \sum_{r=i}^k\binom{n}{i} \binom{k}{r}S(r,i)B(k-r)(f^{\lambda})^2,$$
where $S(r,i)$ denotes the Stirling number of the second kind. 
\end{corollary}

\section{Further combinatorial connections}\label{s-comb}

\subsection{Young's lattice and vacillating tableaux}\label{s-comb1}

Recall that the classical {\em Young's lattice} $\mathbf{Y}$ is 
an infinite unoriented graph whose vertices are all partitions of all 
non-negative integers. The edges are given as follows:
two vertices $\lambda\vdash n$ and 
$\mu\vdash n+1$ are connected by an edge if and only if
$\lambda$ can be obtained from $\mu$ by removing a removable node
(equivalently, $\mu$ can be obtained from $\lambda$ by inserting an
insertable node).

Recall from \cite{HL} that an {\em $n$-vacillating tableau
of semi-length $k$} is a walk on $\mathbf{Y}$ which starts at $(n)$,
has $2k$ steps and at each odd step we remove a removable node
while at each even step we insert an insertable node.
Next, we recall from \cite{CY07} that a {\em simplified 
vacillating tableau of shape $\mu\vdash m$ and semi-length $k$} 
is a walk on $\mathbf{Y}$ which starts at $\varnothing\vdash 0$,
has $2k$ steps and at each odd step we either do nothing or
remove a removable node while at each even step we either do nothing
or insert an insertable node. Note that the definition forces
$m\leq k$. We denote the number of simplified 
vacillating tableau of shape $\mu\vdash m$ and semi-length $k$
by $g_k(\mu)$.

\subsection{Connection to 
$(\Ind^{R_n}_{R_{n-1}}\circ\Res^{R_n}_{R_{n-1}})^k(\mathbb{C}_{\triv})$}\label{s-comb2}

Recall that, for $\lambda\vdash m$, we denote by $f^\lambda$ 
the dimension of the Specht module ${S}^\lambda$, that is, the number
of standard Young tableau of shape $\lambda$.

\begin{lemma}\label{lem-s-comb2.1}
Let $n\geq k$. The we have:
\begin{enumerate}[$($a$)$]
\item \label{lem-s-comb2.1.1}
The length of 
$(\Ind^{R_n}_{R_{n-1}}\circ\Res^{R_n}_{R_{n-1}})^k(\mathbb{C}_{\triv})$
equals $\sum_{\mu}g_k(\mu)$.
\item \label{lem-s-comb2.1.2}
The dimension of 
$(\Ind^{R_n}_{R_{n-1}}\circ\Res^{R_n}_{R_{n-1}})^k(\mathbb{C}_{\triv})$
equals $\sum_{\mu}g_k(\mu)f^\mu\binom{n}{|\mu|}$.
\end{enumerate}
\end{lemma}

\begin{proof}
The branching rule for the rook monoids, see Figure~\ref{fig1},
says that restricting a simple $R_n$-module which corresponds to some
$\lambda$ to $R_n$ decomposes into a multiplicity-free direct sum of
simple modules.  Moreover, the diagrams of simple which do appear are
obtained from $\lambda$ by either doing nothing or removing a 
removable node. Induction works in the dual way. 

Comparing this with the definition of simplified vacillating tableaux,
we see that the multiplicity of a simple module indexed by
$\mu$ in $(\Ind^{R_n}_{R_{n-1}}\circ\Res^{R_n}_{R_{n-1}})^k(\mathbb{C}_{\triv})$
is exactly $g_k(\mu)$.  The latter simple module has dimension
$f^\mu\binom{n}{|\mu|}$. This implies both claims of the lemma.
\end{proof}

By Proposition~\ref{prop:2}, the dimension of 
$(\Ind^{R_n}_{R_{n-1}}\circ\Res^{R_n}_{R_{n-1}})^k(\mathbb{C}_{\triv})$
equals the number of the $n$-restricted partitions of $[n+k]$.
Combining with Lemma~\ref{lem-s-comb2.1}\eqref{lem-s-comb2.1.2},
we obtain a formula how to compute the number of the $n$-restricted 
partitions of $[n+k]$ using $f^\mu$ and $g_k(\mu)$.

\subsection{Grothendieck group}\label{s-comb3}

Let $\mathcal{G}_0(R_n\text{- mod})$ denote the Grothendieck 
group of the category $R_n\text{- mod}$. 
The space $\bigoplus_{n\geq 0}\mathcal{G}_0(R_n\text{- mod})$ is a ring
with multiplication given by an appropriate induction. 
Let $\Lambda$ denote the ring of symmetric functions.  
In~\cite{9}, we prove that 
$$ \bigoplus_{n\geq 0}\mathcal{G}_0(R_n\text{- mod}) \cong \Lambda \otimes
\mathbb{Z}[\mathbb{Z}_{\geq 0}],$$
where $\mathbb{Z}[\mathbb{Z}_{\geq 0}]$ is the monoid algebra
of the monoid $(\mathbb{Z}_{\geq 0},+)$ over $\mathbb{Z}$. This bijection sends the irreducible module $R_n^\lambda$ of $R_n$ to the Schur function $s_\lambda$ in the $n$-th copy of $\Lambda$. 
 
Due to Theorem~\ref{thm:key}, under this bijection, the representation 
$(\Ind^{R_n}_{R_{n-1}}\circ\Res^{R_n}_{R_{n-1}})^k(\mathbb{C}_{\triv})$ 
corresponds to $\displaystyle{\sum_{\mu}g_k(\mu)s_\mu}$. 
Combining \cite[Example~3.18]{Solomon} 
or \cite[Theorem~2.20]{11} with Theorem~\ref{thm:1}, gives
$$\displaystyle{\sum_{\mu}g_k(\mu)s_\mu}= \sum_{i=|\lambda|}^k {k\choose i}B(k-i)\sum_{1\leq|\lambda|\leq n}s(i,|\lambda|)f^\lambda s_{\lambda}. $$
The multiplicity of $s_\lambda$ in the right hand side of the above is 
$\displaystyle\sum_{i=|\lambda|}^k {k\choose i}B(k-i)s(i,|\lambda|)f^\lambda$. 
If we denote by $B(k,|\lambda|)$ the number of set-partitions of $k$ which may 
contain any number of parts but satisfy the
condition that exactly $|\lambda|$ parts are marked (in particular, 
the total number of parts 
must be at least $|\lambda|$), then the latter multiplicity is 
precisely equal to $B(k,|\lambda|)f^\lambda$. Using the fact that the 
multiplicity of $s_\lambda$ in $h_1^k$, where $h_1$ is the complete 
homogeneous symmetric function of degree $1$, is $f^\lambda$, 
we obtain a representation-theoretic interpretation of \cite[Theorem~7]{17}.

\section*{Acknowledgments}
VM is partially supported by the Swedish Research Council. SS is partially supported by DST-INSPIRE faculty research grant DST/INSPIRE/04/2021/000268 and ANRF/ECRG/2024/002319/PMS.


\begin{thebibliography}{ccccccc}

\bibitem[AM]{AM} Andr{}\'e, C.; Martins, I.
Schur-Weyl dualities for the rook monoid: an 
approach via Schur algebras. 
Semigroup Forum {\bf 109} (2024), no. 1, 38--59.

\bibitem[B-Z24]{17} 
Berikkyzy, Z.; Harris, P.; Pun, A.; Yan, C.; Zhao, C.
Combinatorial identities for vacillating tableaux.
Integers {\bf 24A} (2024), Paper No. A4, 36 pp.

\bibitem[C-V19]{1} 
Caicedo, J.; Moll, V.; Ramirez, J.; Villamizar, D.
Extensions of set partitions and permutations.
Electron. J. Combin. {\bf 26} (2019), no. 2, Paper No. 2.20, 45 pp.

\bibitem[C-Y07]{CY07}
Chen, W.; Deng, E.; Du, R.; Stanley, R.; Yan, C.
Crossings and nestings of matchings and partitions.
Trans. Amer. Math. Soc. {\bf 359} (2007), no. 4, 1555--1575.

\bibitem[EL93]{EL93} Easdown D.; Lavers, T. 
The inverse partition semigroup.
Journal of Algebra {\bf 169} (1994), 393--421.

\bibitem[EL95]{EL95}  Easdown D.; Lavers, T.
The dual symmetric inverse monoid.
Semigroup Forum {\bf 50} (1995), 214--224.

\bibitem[FL98]{FL98} FitzGerald, D.; Leech, J. 
Dual symmetric inverse monoids and representation theory. 
J. Austral. Math. Soc. Ser. A {\bf 64} (1998), no. 3, 345--367.

\bibitem[GM09]{GM}
Ganyushkin, O.; Mazorchuk, V.
Classical finite transformation semigroups.
An introduction. Algebra and Applications, {\bf 9}. 
Springer-Verlag London, Ltd., London, 2009. xii+314 pp.

\bibitem[Gr02]{Grood}
Grood, C. Specht module analog for the rook monoid. 
Electron. J. Combin. {\bf 9} (2002), no. 1, Research Paper 2, 10 pp.

\bibitem[Hal01]{6} Halverson, T.
Characters of the partition algebras.
J. Algebra {\bf 238} (2001), no. 2, 502--533.

\bibitem[Hal04]{10} Halverson, T.
Representations of the {$q$}-rook monoid.
J. Algebra {\bf 273} (2004), no. 1, 227--251.

\bibitem[HL04]{HL} Halverson, T.; Lewandowski, T.
RSK insertion for set partitions and diagram algebras.
Electron. J. Combin. {\bf 11} (2004/06), no. 2, Research Paper 24, 24 pp.

\bibitem[HR05]{13} Halverson, T.; Ram, A. 
Partition algebras. 
European J. Combin. {\bf 26} (2005), no. 6, 869--921.

\bibitem[Jo94]{Jones} Jones. V.
The Potts model and the symmetric group.
Unpublished notes. 1994.

\bibitem[KM08a]{5} Kudryavtseva, G.; Mazorchuk, V.
Schur-{W}eyl dualities for symmetric inverse semigroups.
J. Pure Appl. Algebra {\bf 212} (2008), no. 8, 1987--1995.

\bibitem[KM08b]{KMb} Kudryavtseva, G.; Mazorchuk, V.
Partialization of categories and inverse braid-permutation monoids.
Internat. J. Algebra Comput. {\bf 18} (2008), no. 6, 989--1017. 

\bibitem[KM09a]{KM09} Kudryavtseva, G.; Mazorchuk, V.
Combinatorial Gelfand models for some semigroups and 
$q$-rook monoid algebras. Proc. Edinb. Math. Soc. (2) {\bf 52}
(2009), no. 3, 707--718.

\bibitem[KM09b]{KM09b} Kudryavtseva, G.; Mazorchuk, V.
On three approaches to conjugacy in semigroups
Semigroup Forum {\bf 78} (2009), no. 1, 14--20.

\bibitem[MSS11]{4} Mansour, T.; Schork, M.; Shattuck, M.
On a new family of generalized Stirling and Bell numbers.
Electron. J. Combin. {\bf 18} (2011), no. 1, Paper 77, 33 pp.

\bibitem[Ma94]{Martin} Martin, P. 
Temperley-Lieb algebras for nonplanar statistical mechanics~--
the partition algebra construction. J. 
Knot Theory Ramifications {\bf 3} (1994), no. 1, 51--82.

\bibitem[Ma96]{Martin96} Martin, P. 
The structure of the partition algebras. 
J. Algebra {\bf 183} (1996), no. 2, 319--358.

\bibitem[MS22]{9} Mazorchuk, V.; Srivastava, S.
Jucys-{M}urphy elements and {G}rothendieck groups for 
generalized rook monoids.
J. Comb. Algebra {\bf 60} (2022), no. 1-2, 185--222.

\bibitem[MS24]{MS24} Mazorchuk, V.; Srivastava, S.
Multiparameter colored partition category and the 
product of the reduced Kronecker coefficients. 
J. Pure Appl. Algebra {\bf 228} (2024), no. 3, Paper No. 107524, 60 pp.

\bibitem[MS25]{MS25} Mazorchuk, V.; Srivastava, S.
Kronecker coefficients for (dual) symmetric inverse semigroups. 
J. Aust. Math. Soc. 118 (2025), no. 1, 65--90.

\bibitem[Me11]{3} Mez{\"o}, I. The $r$-Bell numbers
J. Integer Seq. {\bf 14} (2011), no. 1, Article 11.1.1, 14 pp.

\bibitem[MR17]{2} Mez{\"o}, I.; Ram{\'\i}rez, J.
Divisibility properties of the $r$-Bell numbers and polynomials.
J. Number Theory {\bf 177} (2017), 136--152.
 
 
\bibitem[MS21]{11} Mishra, A.; Srivastava. S.
Jucys-{M}urphy elements of partition algebras for the rook monoid.
Internat. J. Algebra Comput. {\bf 31} (2021), no. 5, 831--864.
 
\bibitem[Pa06]{Pa} Paget, R. 
Representation theory of $q$-rook monoid algebras. 
J. Algebraic Combin. {\bf 24} (2006), no. 3, 239--252.
 
 
\bibitem[Sa01]{Sa}  Sagan, B. 
The symmetric group. Representations, combinatorial 
algorithms, and symmetric functions. Second edition. 
Graduate Texts in Mathematics, {\bf 203}. Springer-Verlag, New York, 2001.   
 
\bibitem[Sc27]{Schur} Schur, I.
{\"U}ber eine Klasse von Matrizen, die sich einer gegebenen 
Matrix zuordnen lassen. Sitzungsberichte der Berliner 
Mathematischen Gesellschaft {\bf 220} (1923), 9--20.
 
\bibitem[So02]{Solomon} Solomon, L. 
Representations of the rook monoid. 
J. Algebra {\bf 256} (2002), no. 2, 309--342.

\bibitem[St16]{8} Steinberg, B.
Representation theory of finite monoids.
Universitext. Springer, Cham, 2016, xxiv+317 pp.

\bibitem[Wa52]{Wagner} Wagner, V.
Generalized groups.
Doklady Akademii Nauk SSSR {\bf 84} (1952), 1119--1122.

\bibitem[We39]{Weyl} Weyl, H.
The Classical Groups. Their Invariants and Representations. 
Princeton University Press, Princeton, NJ, 1939. xii+302 pp.

\end{thebibliography}
\end{document}